\documentclass[11pt,reqno]{amsart}

\usepackage{amssymb}
\usepackage{amscd}
\usepackage{amsfonts}
\usepackage{mathrsfs}
\usepackage{version}
\usepackage{amsthm}

\numberwithin{equation}{section} \theoremstyle{plain}

\newtheorem{theorem}{Theorem}[section]
\newtheorem{proposition}[theorem]{Proposition}
\newtheorem{lemma}[theorem]{Lemma}
\newtheorem{corollary}[theorem]{Corollary}

\theoremstyle{remark}

\newtheorem{remark}[theorem]{Remark}

\newtheorem*{mainthm3-repeat}{Theorem \ref{mainthm3}}
\newtheorem*{sec3-key-prop-rpt}{Theorem \ref{sec3-key-prop}}

\renewcommand{\leq}{\leqslant}
\renewcommand{\geq}{\geqslant}
\newsavebox{\proofbox}
\savebox{\proofbox}{\begin{picture}(7,7)  \put(0,0){\framebox(7,7){}}\end{picture}}

\newcommand\F{\mathbb{F}}

\begin{document}

\title{Topological generation of special linear groups}

\author[Gerhardt]{Spencer Gerhardt}
\address{Department of Mathematics \\
University of Southern California\\
3620 S. Vermont Ave\\
Los Angeles CA 90007}
\email{sgerhard@usc.edu}

\begin{abstract}

Let $C_1,\ldots,C_e$ be noncentral conjugacy classes of the algebraic group $G=SL_n(k)$ defined over a sufficiently large field $k$, and let $\Omega:=C_1\times \ldots \times C_e$. This paper determines necessary and sufficient conditions for the existence of a tuple $(x_1,\ldots,x_e)\in\Omega$ such that $\langle x_1,\ldots,x_e\rangle$ is Zariski dense in $G$.  As a consequence, a new result concerning generic stabilizers in linear representations of algebraic groups is proved, and existing results on random $(r,s)$-generation of finite groups of Lie type are strengthened.

\end{abstract}

\maketitle

\setcounter{tocdepth}{1}

\section{Introduction and statement of results}\label{sec1}

Let $C_1,\ldots, C_e$ be conjugacy classes of a group $G$, and $\Omega:= C_1\times\ldots \times C_e$. The  problem of specifying conditions for tuples $\omega\in\Omega$ which generate $G$, or which generate a subgroup of $G$ with special properties, arises in many different contexts. For instance, if  $C_1,\ldots ,C_e$ are conjugacy classes of $G=SL_n(\mathbb{C})$, the Deligne-Simpson problem asks for necessary and sufficient conditions for the existence of a tuple  $(x_1,\ldots,x_e)\in\Omega$ such that $x_1\cdots x_e=1$ and $\langle x_1,\ldots,x_e\rangle$ acts irreducibly on the natural module.  A  solution to this question yields information concerning monodromy groups of regular systems of differential equations on $\mathbb{C}P^1$ (see \cite{kostov5} for further details), and has been studied in numerous settings and generalizations \cite{boevey} \cite{katz} \cite{kostov1} \cite{kostov2} \cite{kostov3} \cite{kostov4} \cite{simpson}. In  \cite{boevey2} and \cite{shaw} necessary and sufficient conditions are given for the existence of such a tuple, although work on the many generalizations remains. 

In a different direction, any finitely generated group which is generated by conjugates of a single element arises as a quotient of a knot group (see  \cite{johnson} for additional details), and such groups have been studied in a variety of contexts \cite{chiodo} \cite{kim} \cite{eisenmann} \cite{osin} \cite{thom}. In the special case where $G$ is a finite simple or almost simple group, the number of conjugates of a fixed element required to generate $G$ is  examined in \cite{guralnick5} and \cite{hall}.

Finally, a good deal of recent work concerns the notion of invariable generation \cite{gelander} \cite{gelander2} \cite{gelander3} \cite{kantor2} \cite{kantor3} \cite{wiegold1} \cite{wiegold2}, in which every tuple in a product of conjugacy classes generates $G$. In particular, the question of whether linear groups such as $G=SL_n(\mathbb{Q})$ or  $G=SL_n(\mathbb{Z})$,  $n\geq3$, are  invariably generated has attracted considerable attention.

In this paper we examine another class generation problem for linear algebraic groups of the form  $G=SL_n(k)$, where $k$ is initially taken to be an uncountable algebraically closed field of arbitrary  characteristic. Once the desired result has been established, we then show the field conditions can be relaxed considerably (to arbitrary fields of characteristic zero, and algebraically closed fields of positive characteristic which are not algebraic over a  finite field). Finally, we see the solution to this problem has interesting applications to generic stabilizers in linear representations of algebraic groups, and random generation of finite groups of Lie type. 

Let us briefly describe the set up of the paper. Assume $G$ is a simple algebraic group defined over an algebraically closed field $k$ which is not algebraic over a finite field. 
Set $\Omega= C_1\times \ldots \times C_e$ and say  $(x_1,\ldots,x_e)\in\Omega$ \textit{generates $G$ topologically} if $\langle x_1,\ldots,x_e\rangle$ is Zariski dense in $G$. Since $G$ is infinitely generated, the relevant finitary notion of generation is topological. In what follows, we  determine necessary and sufficient conditions for the existence of a tuple $\omega\in\Omega$ which topologically generates $G=SL_n(k)$. In joint work with Burness and Guralnick \cite{burness} \cite{burness1}, conditions for the remaining simple algebraic groups are determined. 

  Recall a conjugacy class $C$ of $G$ is called $\it{quadratic}$ if a representative of  $C$ has a degree two minimal polynomial. Our main result is as follows. 

\begin{theorem}\label{main-thm}  Let $C_1,\ldots,C_e$ be noncentral conjugacy classes of the algebraic group $G=SL_n(k)$ where $k$ is an uncountable algebraically closed field, and $n\geq 3$. Let $\gamma_i$ be the dimension of the largest eigenspace of a representative of $C_i$ on the natural module for $G$. Then there is a tuple $\omega\in\Omega$ topologically generating $G$ if and only if the following conditions hold:

\begin{itemize} 
\item[($i$)] $\sum_{i=1}^e\gamma_i\leq n(e-1)$; 
\item[$(ii)$] it is not the case that $e=2$ and $C_1,C_2$ are quadratic. 
\end{itemize} 

\end{theorem} 
Let us make a few comments about the theorem. First, the assumption that $n>2$ is made for  purposes of uniformity. Theorem 4.5 below provides the relevant necessary and sufficient conditions for $SL_2(k)$. Second, it follows from the above theorem that for any noncentral conjugacy class $C$ of $G=SL_n(k)$, $n\geq3$, it is possible to generate $G$ topologically by choosing $n$ elements from $C$. Furthermore, this bound is sharp. If $C$ is a class of transvections, it is easy to see that at least $n$ conjugates in $C$ are required to generate $G$ topologically.  
\\

The notion of a generic property underlies much of the reasoning given below, so let us fix its meaning in the present context. For any irreducible variety $X$ defined over an algebraically closed field $k$, say a subset of $X$ is \textit{generic} if it contains the complement of a countable union of proper closed subvarieties of $X$.  Note over any uncountable algebraically closed field a generic subset has $k$-points.

In the statement of Theorem 1.1 it is assumed that $G=SL_n(k)$ is defined over an uncountable algebraically closed field $k$. It is evident some condition on $k$ is required, as $G$ is locally finite when $k$ is algebraic over a finite field. However once Theorem 1.1 has been established, it is possible to use this result to substantially relax the assumptions placed on $k$. Furthermore, the existence of a single tuple $\omega\in\Omega$ topologically generating $G$ can be used to establish that some larger subset of tuples in $\Omega$ generate $G$ topologically.

\begin{theorem} Let $C_1,\ldots, C_e$ be noncentral conjugacy classes of $G=SL_n(k)$, $n\geq3$, where $k$ is any field. Assume  
$\sum_{i=1}^e\gamma_i\leq n(e-1)$, and it is not the case that $e=2$ and $C_1,C_2$ are quadratic.  
\begin{itemize} 
\item[$(i)$] If $k$ is an uncountable algebraically closed field, a generic subset of tuples in $\Omega$ generate $G$ topologically.
\item[$(ii)$] If $k$ is any field of characteristic zero, a nonempty open subset of tuples in $\Omega$ generate $G$ topologically.
\item[$(iii)$] If $k$ is an algebraically closed field that is not algebraic over a finite field, a dense subset of tuples in $\Omega$ generate $G$ topologically.
\end{itemize} 
\end{theorem} 

Similarly if $G$ is any simply connected simple algebraic group defined over a field $k$, and $\omega\in\Omega$ is a tuple that topologically generates $G$, natural analogues to conditions $(i)-(iii)$  exist  (see Theorem 6.2 below). 
\newline

 Now let us briefly describe the proof strategy for the main theorem. To begin, pick noncentral conjugacy classes $C_1,\ldots,C_e$ of $G=SL_n(k)$. Let $M$ be a closed subgroup of $G$, and 

$$Y_M:=\overline{\{(x_1,\ldots,x_e)\in\Omega\mid \langle{x_1,\ldots,x_e}\rangle\subset M^g \text{, for some } g\in G\}}.$$
Up to conjugacy, there exist finitely many maximal positive-dimensional closed subgroups of $G$ (see Corollary 3 of \cite{liebeck6}). Call these subgroups\newline $M_1,\ldots,M_t$. In positive characteristic, a \emph{subfield subgroup} of $G$ is a finite subgroup of the form $G(q)$ with $q=p^a$  where $p$ is the characteristic of the underlying field $k$. Note each subfield subgroup $G(q)$ is a closed subgroup of $G$, and every proper closed subgroup of $G$ is contained in either some subfield subgroup, or some positive dimensional maximal subgroup of $G$ (see Lemma 3.1 of \cite{guralnick}).

From now on we will say that generic tuples in $\Omega$ possess a given property if a generic subset of $\Omega$ has the property. The overall strategy is as follows. Suppose we can show that generic tuples in $\Omega$ generate an infinite group, and that generic tuples in $\Omega$ are not contained in $\bigcup_{i=1}^t Y_{M_i}$.  Then generic tuples will generate a subgroup of $G$ contained in no conjugate of a maximal positive-dimensional closed subgroup, or a conjugate of any subfield subgroup. In this case, generic tuples in $\Omega$  topologically generate $G$. 

The argument proceeds by induction on the dimension of the natural module. In the base case $n=3$, the maximal subgroup structure of $G$ is completely understood. Here, for any closed subgroup $M$ of $G$, Lemma 2.1 provides the useful bound 
$$\dim Y_M\leq \dim  \prod_{j=1}^e \ (C_j\cap M)+ \dim   G/M.$$
Then by computing the dimensions of varieties, we show that 
$$ \dim  Y_M\leq  \dim \prod_{j=1}^e\ (C_j\cap M)+ \dim  G/M < \dim   \Omega $$ 
for each maximal closed subgroup, and each subfield subgroup $M$ of $G$. It follows that generic tuples in $\Omega$ generate an infinite group not contained in $\bigcup_{i=1}^t Y_{M_i}$.

\par 
For the inductive argument, slightly more indirect reasoning is required. Using the inductive hypothesis, we infer that some tuple $\omega\in\Omega$ topologically generates a large rank subgroup of $G$. Making use of this fact, we then show that generic tuples in $\Omega$
generate a group having a list of properties that no subfield subgroup nor positive-dimensional maximal subgroup can share. In particular, we show that generic tuples in $\Omega$ 
act irreducibly and primitively on the natural module, and contain ``strongly regular" elements of infinite order. (For a given maximal torus $T$ of $G$, an element of $G$ is said to be \textit{strongly regular} if it has distinct eigenvalues on the root spaces of $Lie(G)$ with respect to $T$). We then show this implies that generic tuples in $\Omega$ topologically generate $G$. 
\ \\

Finally we turn to applications of the main theorem. The first concerns linear representations of algebraic groups. In what follows let $G$ be a simply connected simple algebraic group defined over an algebraically closed field $k$ of characteristic $p \geq 0$. Unless otherwise stated,  let $V$ be an irreducible finite dimensional rational $kG$-module.  Recall for any $x\in V$ the point stabilizer of $x$ in $G$ is written $G_x:=\{g\in G \mid gx=x\}$. A subgroup $H$ is said to be a $\textit{generic stabilizer}$ of $G$ on $V$ if there exists a nonempty open subset $V_0\subset V$ such that $G_x$ is a conjugate of $H$ for every $x\in V_0$. 

In characteristic zero, it is a well-known result of Richardson \cite{richardson} that a generic stabilizer exists whenever $V$ is a smooth affine variety (in fact, this result holds when $G$ is reductive). In positive characteristic this fails, although in many cases of interest generic stabilizers do exist, for instance when $V$ is irreducible \cite{guralnick2}.

The structure of generic stabilizers, and in particular the question of whether a generic stabilizer is finite or trivial, arises frequently in the context of invariant theory.  In characteristic zero, the structure theory of generic stabilizers has been carefully investigated. For instance, \`Ela\v svili \cite{elasvili} and Popov  \cite{popov1} classify when a generic stabilizer is finite or trivial.  A well-known related result states that a generic stabilizer of $G$  
is nontrivial if its algebra of invariant polynomials $k[V]^G$ is free (see Theorem 8.8 of \cite{popov3}). For semisimple algebraic groups, the question of when a generic stabilizer is trivial is studied in \cite{elasvili}, \cite{elasvili2} and \cite{popov2}.

 Much recent work concerns the structure of generic stabilizers in arbitrary characteristic. In this context, Guralnick and Lawther  \cite{guralnick2} recently established that a generic stabilizer is finite if and only if $\dim V> \dim G$.  
 
 In addition, for applications to Galois cohomology and essential dimension, it is often necessary to consider a generic stabilizer as a group scheme \cite{brosnan} \cite{garibaldi4} \cite{karpenko} \cite{lot}. Further recent work  \cite{garibaldi2} \cite{garibaldi3} shows if $\dim  V> \dim  G$, then in most cases a generic stabilizer is trivial as a group scheme. 

As these results relate to the bound given in  Theorem 1.3 below, let us describe this work in slightly greater detail. If $G$ is an algebraic group defined over a field $k$ and $\rho: G\rightarrow GL(V)$ is a representation, $G$ is said to act \textit{generically freely} on $V$ if there exists a dense open subset $U$ of $V$ such that for every extension $K$ of $k$ and every $u\in U(K)$, the stabilizer $G_u$ (a closed subgroup-scheme of $G\times K$) is the group scheme 1. In \cite{garibaldi2} and \cite{garibaldi3} it is shown if $\rho:G\rightarrow GL(V)$ is a faithful, irreducible representation of a simple algebraic group defined over any algebraically closed field $k$, then a generic stabilizer acts generically freely on $V$ if and only if $\dim V> \dim G$ and the pair $(G,V)$ does not appear on a short list of exceptional cases (see Tables 4 and 5 in  \cite{garibaldi}). 

For applications, it is often necessary to know not just whether $G_v$ is infinitesimal (i.e. 1 as a group scheme), but whether $G_v$ is trivial as a group scheme. To determine this, in \cite{garibaldi} it is established when a generic Lie algebra stabilizer $Lie(G)_v$ is as small as possible (i.e. when $Lie(G)_v= \text{ker } d\rho$). 

More precisely, let
$$V^{G}=\{v\in V|\ gv=v \text{ for all } g\in G\}$$ 
and for any Lie subalgebra  $\mathfrak{s}\subseteq \mathfrak{g}$, let 
$$V^{\mathfrak{s}}=\{v\in V|\ d\rho(x)v=0 \text{ for all } x\in\mathfrak{s}\}.$$ 
Furthermore let $\rho: G\rightarrow  GL(V)$ be a representation of $G$ where $V$ has a $G$-subquotient $X$ with  $X^{[\mathfrak{g},\mathfrak{g}]}$ $= 0$. In \cite{garibaldi} new bounds $b(G)$ are determined such that if  $\dim  X > b(G)$, then for generic $v\in V$, $Lie(G)_v = \text{ker } d\rho$.

\par 
In Section 5, we use Theorem 1.1 to establish related bounds for algebraic groups, as opposed to their Lie algebras. Let $G=SL_n(k)$ where $k$ is an algebraically closed field of arbitrary characteristic, and let $V$ be a finite dimensional rational $kG$-module such that $V^G=0$. (Note in particular we are not assuming that $V$ is irreducible). In this context, we show a generic stabilizer $G_v$ of $G$ is trivial as a group whenever $\dim V>\alpha$. Interestingly this bound $\alpha$ is the same one obtained for the corresponding Lie algebra $\mathfrak{sl}_n(k)$ in \cite{garibaldi}.  

Furthermore, our result relates to a generic stabilizer being trivial as a group scheme. To establish this fact it suffices to show  the group stabilizer is generically trivial, and the generic Lie algebra annihilator is zero. (In characteristic zero, it suffices to show the group stabilizer is generically trivial.) Our result is as follows.

\begin{theorem}\label{gen-stab}
Let $G=SL_n(k)$, where $n\geq3$ and $k$ is an algebraically closed field. Assume $V$ is a finite dimensional rational $kG$-module such that $V^G=0$. If $\dim V> \frac{9}{4} n^2$, then a generic stabilizer is trivial. 
\end{theorem} 

The bound in Theorem 1.3 is close to best possible. For instance, if $V = {\rm Sym}^2(W) \oplus {\rm Sym}^2(W)$, where $W$ is the natural module for $G$, then $\dim V = n(n+1)$ and a generic stabilizer (in characteristic $p\ne 2$)
is a nontrivial elementary abelian group. The generic stabilizer is the generic intersection of the
individual generic stabilizers, which are  centralizers of graph automorphisms of order 2  inverting a maximal torus of $G$. It follows the generic stabilizer is a nontrivial elementary abelian 2-group
contained in the maximal torus $T$ of $G$ inverted by both of these involutions (see \cite{burness2}, Lemma 3.6). 

Note in the above example that the module $V$ is reducible. It appears that examples where $\dim V> n^2$ and a generic stabilizer is nontrivial arise only when $V$ is reducible. However when $V$ is irreducible, there are still instances where $\dim V < \frac{9}{4} n^2$ and a generic stabilizer is nontrivial. For example, if $G=SL_9(k)$ and $V=\wedge^3(W)$ where $W$ the natural module for $G$. Then $\dim V=84\leq \frac{9}{4}\cdot 9^2$, and a generic stabilizer (in characteristic $p\ne 3$) is a nontrivial elementary abelian group. See Table 1 of \cite{guralnick2} for this and other examples. 
\\
\\
 
 A second application of our main theorem concerns a classical group-theoretic problem: the random generation of finite simple groups. The following two questions date back to the 19th century: 
\begin{itemize} 
\item[$(i)$] does the probability that two random elements of $\mathrm{Alt}_n$ generate $\mathrm{Alt}_n$ tend to 1 as $n\rightarrow\infty?$

\item[$(ii)$] which finite simple groups appear as quotients of the modular group $PSL_2(\mathbb{Z})$? 
\end{itemize} 
The first question was posed by Netto in 1892 and answered affirmatively by Dixon \cite{dixon} in 1969. In the same paper, Dixon  conjectures that for \textit{any} finite simple group $H$, the probability that two random elements of $H$ generate $H$ tends to 1 as the size of $H$ tends to infinity. After some effort, this conjecture was proven by Kantor-Lubotzky \cite{kantor} and Liebeck-Shalev \cite{liebeck} in the 1990s using probabilistic methods. 
  \par The positive solution of Dixon's conjecture led to many more fine-grained questions concerning the random generation of finite simple groups. A natural refinement of the above question concerns  whether two random elements of prime orders $r$ and $s$ generate $H$ as $|H|\rightarrow\infty$. The property that 
  two elements of orders $r$ and $s$ generate $H$ with probability tending to 1 as $|H|$ tends to infinity
  is called $\emph{random $(r,s)$-generation}$.  Perhaps surprisingly given the strength of the statement, this property has also been found to hold in many cases \cite{liebeck2}, although counterexamples exist \cite{liebeck3}, and some questions remain. In particular, by work of Liebeck and Shalev \cite{liebeck2} it is known that finite simple classical groups have random $(r,s)$-generation when $(r,s)\neq(2,2)$ and the rank of the group is large enough (depending on the primes $r$ and $s$).
  \par 
 The weaker question of determining whether \textit{any} elements of orders $r$ and $s$, and in particular whether any elements of orders 2 and 3, generate a finite simple group is also a classical group-theoretic problem. This property is called \textit{$(r,s)$-generation}. Its geometric background and long history in the literature are discussed in \cite{liebeck3}. 
 
 Although it is not immediately obvious, question $(ii)$ listed above concerns the $(2,3)$-generation of finite simple groups. As $PSL_2(\mathbb{Z})$ is isomorphic to the free product $C_2\ast C_3$, the question may be rephrased in the following form: which finite simple groups are $(2,3)$-generated? 
  
  Much progress on this problem has been made in recent years. In \cite{lubeck2}, it is shown all finite simple exceptional groups (except Suzuki groups) are $(2,3)$-generated. Suzuki groups are obvious exceptions because they do not contain any elements of order $3$. For the classical groups, recent advances have been made by proving stronger random $(2,3)$-generation results. In \cite{liebeck3} it is shown  that all finite simple classical groups not of the form $PSp_4(q)$ have random $(2,3)$-generation. Consequently, aside from $PSp_4(q)$, all but finitely many finite simple classical groups appear as quotients of $PSL_2(\mathbb{Z})$. 
  
  For groups of type $PSp_4(q)$, the random $(2,3)$-generation results are more nuanced. First, groups of the form $PSp_4(p^f)$ with $p \in \{2,3\}$ are not $(2,3)$-generated for any $f$. For $PSp_4(p^f)$, with $p\geq5$, the probability that two random elements of orders 2 and 3 generate $PSp_4(p^f)$ tends to $\frac{1}{2}$ as the size of the group goes to infinity \cite{liebeck3}.  In recent work \cite{guralnick3} it is shown that all exceptional groups of Lie type (except Suzuki groups) also have random $(2,3)$-generation. 

\par Using random $(2,3)$-generation as a guide, one might expect (modulo a few small rank exceptions) that random $(r,s)$-generation holds for all finite simple groups, with $r,s$ independent of the rank of the group. In Section 5, we use topological generation to prove this fact for linear and unitary groups, hence strengthening results in \cite{liebeck2}. In work with collaborators \cite{burness} \cite{burness1}, this property is established for the remaining finite simple groups. 

In what follows, the rank of a finite unitary group $G(q)$ is defined to be its untwisted rank, which coincides with the rank of the ambient simple algebraic group $G$. Furthermore, the characteristic of $\mathbb{F}_q$ is assumed to be  fixed. 

\begin{theorem}\label{finite}  

Fix a positive integer $n$, and let $r,s$ be primes with $(r,s)\neq (2,2)$. Let $G(q)$ be a finite simple linear or unitary group of fixed rank $n$ over the finite field $\F_q$ and assume $G(q)$  contains elements of orders $r$ and $s$. Let $x$ and $y$ be random elements of orders $r$ and $s$ in $G(q)$. Then the probability that $x$ and $y$ generate $G(q)$ tends to $1$ as $q \rightarrow \infty$.
\end{theorem}

\newpage

\textsc{Acknowledgments.}  This research was partially supported by NSF grant DMS-1302886. The author would like to thank Robert Guralnick for many helpful comments and guidance during the development of this paper. In addition, thanks are due to Gunter Malle for numerous valuable remarks on an earlier draft of this article, and to the anonymous referee whose careful reading led to several organizational changes and improvements in the text.

\section{Preliminary results}\label{sec 2}

Before proving Theorem 1.1, it will be helpful to establish some general facts about algebraic groups required in our arguments. This background material will be stated in slightly greater generality than necessary, so that it may be applied in future work. 

In this section (unless otherwise stated) let $G$ be a simply connected simple algebraic group defined over an algebraically closed field of characteristic $p\geq 0$. When $G$ is a classical group, let $V$ be the natural module for $G$. Finally, let $C_1,\ldots,C_e$ be noncentral conjugacy classes of $G$, and $\Omega:=C_1\times\ldots\times C_e$. Say that a tuple $(x_1,\ldots,x_e)\in G^e$ has a given group property if the closure of $\langle x_1,\ldots,x_e\rangle$ has that property. 

For any closed subgroup $M$ of $G$, let:

\begin{itemize} 
\item[-] $\Delta:= \prod_{j=1}^e (C_j\cap M)$
\item[-] $X:=
\bigcup_{g\in G}\ (M^g\times \ldots\times M^g)\subset G^e$
\item[-] $\varphi: G\times \Delta \rightarrow X$ be the map $(g, (x_1,\ldots,x_e))\mapsto (x_1^g,\ldots,x_e^g)$ \item[-] $Y_M:=\overline{\{(x_1,\ldots,x_e)\in\Omega\mid \langle{x_1,\ldots,x_e}\rangle\subset M^g \text{, for some } g\in G\}}$

\end{itemize}

Note for any closed subgroup $M$ of $G$,  $(x_1^g,\ldots, x_e^g)\in im(\varphi)$ if and only if $(x_1,\ldots,x_e)\in\Omega$ and $\langle x_1,\ldots,x_e\rangle\subset M^g$ for some $g\in G$. In particular,  $\overline{im(\varphi)}=Y_M$. 

There exist finitely many conjugacy classes of maximal positive-dimensional closed subgroups of $G$ (see Corollary 3 in  \cite{liebeck6}). Call these subgroups $M_1,\ldots,M_t$. As discussed in Section 1, the overall strategy is to prove that generic tuples in $\Omega$ topologically generate an infinite group and that generic tuples in $\Omega$ are not contained in $\bigcup_{i=1}^t Y_{M_i}$. The next result reduces proving the base case $n=3$ of the main theorem to performing calculations on the dimensions of explicit varieties. 

\begin{lemma} Let $M$ be a closed subgroup of $G$. We have $$\dim   Y_M\leq  \dim  \Delta  + \dim  G/M.$$ 
 \end{lemma}

\begin{proof}
To begin, note any $(x_1,\ldots,x_e)^g\in im(\varphi)$ will be in the same orbit as $(x_1,\ldots,x_e)\in\Delta$, and hence their fibers will have the same dimension. So without loss of generality, assume $(x_1,\ldots,x_e)\in\Delta$ and consider the map $\varphi': M\times \Delta\rightarrow \Delta$,  
$(g, (x_1,\ldots,x_e))\mapsto (x_1^g,\ldots,x_e^g)$. Clearly $im(\varphi')=\Delta$, and hence every fiber of $\varphi$ has dimension at least $\dim M$. By the fiber theorem (see Corollary 4 of EGA III \cite{grothendieck}), $\dim  Y_M\leq \dim \Delta +\dim  G/M$.\end{proof}

\begin{remark}
Over an uncountable algebraically closed field, the above lemma reduces proving generic tuples in $\Omega$ topologically generate $G$ to establishing that $\dim \Omega> \dim  \Delta +\dim G/M$, for each maximal positive-dimensional closed subgroup $M$ of $G$.
\end{remark}

The following is a version of Burnside's Lemma for algebraic groups, which allows us to recast Lemma 2.1 in a more geometric light. 

\begin{lemma} Consider the action of $G$ on the coset variety $G/M$. Then for $x\in M$, $$ \dim   G/M - \dim\  (G/M)^x= \dim   C - \dim\ (C\cap M)$$ where $(G/M)^x$ is the set of fixed points of $x$, and $C$ is the conjugacy class of $x$ in $G$. 
\end{lemma} 
\begin{proof} 
 This is Proposition 1.14 in \cite{lawther}. 
 \end{proof}
 
This immediately yields the following useful bound on fixed point spaces.

\begin{lemma}
 Write $C_i=x_i^G$ for $i=1,\ldots, e$. If $\Omega\subset Y_M$, then
 
 $$\sum_{j=1}^e\dim\  (G/M)^{x_j}\geq (e-1)\cdot  \dim  G/M.$$  
 \end{lemma} 
  \begin{proof} 
 Assume $\Omega\subset Y_M$. By Lemma 2.1, we have $$\dim  \Omega\leq \sum_{j=1}^e \dim\  (C_j\cap M) + \dim  G/M.$$ Summing from 1 to $e$, Lemma 2.3 yields  
 $$ \dim\  \Omega  =  \sum_{j=1}^e \dim\  G/M- \sum_{j=1}^e \dim\  (G/M)^{x_j}+ \sum_{j=1}^e \dim\
 (C_j\cap M),$$ and by combining these facts we deduce that $$\sum_{j=1}^e \dim\  (G/M)^{x_j}\geq (\sum_{j=1}^e \dim  G/M) -\dim  G/M.$$
 \end{proof}

For the inductive argument required in the proof of Theorem 1.1, we will use the inductive hypothesis to infer the existence of a tuple $\omega\in\Omega$ which topologically generates a large rank subgroup of $SL_n(k)$. We will then use this fact to show that generic tuples in $\Omega$ generate a group having a list of properties that no subfield subgroup nor positive-dimensional maximal subgroup can share. 

Typically these arguments proceed in two stages. The first is to prove the existence of a tuple $\omega\in\Omega$ having some desirable property, and the second is to show that this property is either an open or generic condition. 

We will make use of several facts from the literature. Let $V$ be a rational $kG$-module. Furthermore let  $\mathcal{R}_{d,e}(V)$ be the set of tuples in $G^e$ fixing a given \textit{d}-dimensional subspace of $V$,  $\mathcal{I}_{d,e}(V)$ be the set of tuples in $G^e$ that fix no \textit{d}-dimensional subspace of $V$, and $\mathcal{I}_e(V)$ be the set of tuples in $G^e$ that act irreducibly on $V$. The following lemma is well-known, and is used to show an open subset of tuples in $\Omega$ generate a group acting irreducibly on the natural module. 

\begin{lemma}
Let $V$ be a rational $kG$-module. Then
\begin{itemize} 
\item[$(i)$] $\mathcal{R}_{d,e}(V)$ is a closed subvariety of $G^e$.
\item[$(ii)$] $\mathcal{I}_{d,e}(V)$ is an open subvariety of $G^e$.
\item[$(iii)$] $\mathcal{I}_e(V)$ is an open subvariety of $G^e$. 

\end{itemize}
\end{lemma} 
\begin{proof}
 Parts $(i)$ and $(ii)$ follow from Lemma 11.1 in \cite{guralnick4}. For part $(iii)$, assume $\dim  V=n$. Then $\cup_{1\leq d\leq n}\mathcal{I}_{d,e}(V)= \mathcal{I}_e(V)$.\end{proof}

An extension of Lemma 2.5 proved in \cite{breuillard} will also be useful for establishing properties of topological generation. Let $q$ be a power of a prime $p$. Recall for a given Frobenius endomorphism $F_q$ of $G$, that the fixed point group $G^{F_q}=G(q)$ is a finite group of Lie type over $\mathbb{F}_q$.

\begin{lemma} 
 There exists a finite collection $V_1,\ldots, V_t$ of irreducible finite-dimensional $kG$-modules with the property that if a proper closed subgroup of $G$ acts irreducibly on $V_i$ for each $i$, then it is finite and conjugate to $G(q)$ for some $q$. In particular, if $char(k) = 0$, there are no such subgroups.
\end{lemma} 

\begin{proof} 
This is Lemma 4.2 in \cite{breuillard}.
\end{proof} 

\begin{remark}
By Lemma 2.5 $(iii)$ acting irreducibly on a single (and hence any finite collection) of rational $kG$-modules is an open condition, so Lemma 2.6 shows the set of tuples in $G^e$ topologically generating a group containing a conjugate of some fixed subfield subgroup $G(q_0)$ of $G$ is open in $G^e$. In $char(k) = 0$, Lemmas 2.5 and 2.6 show that topologically generating $G$ is an open condition. 
\end{remark} 

\par 
To establish the remaining generic properties of $\Omega$ it will be necessary to move between properties holding generically on a simple algebraic group $G$ and properties holding generically on $G^e$. In what follows, let $F_e$ refer to the free group of rank $e$. 

\begin{lemma}
 For any nontrivial word $w\in F_e$, the word map $\varphi_w:G^e\rightarrow G$ which sends  $(x_1,\ldots,x_e)\mapsto w(x_1,\ldots,x_e)$ is dominant. In particular,  $Z \subseteq G$ is a generic subset of $G$ if any only if  $$\{(x_1,\ldots,x_e)\in G^e\ |\ w(x_1,\ldots,x_e)\in Z\}$$ is a generic subset of $G^e$. 
\end{lemma}
\begin{proof} 
This is a result of Borel, and appears as Proposition 2.5 in  \cite{breuillard}. 
\end{proof} 

Let us say a tuple $(x_1,\ldots,x_e)\in G^e$ 
possesses a property of an element of $G$ (has infinite order, is regular semisimple, etc.) if there is a nontrivial word $w\in F_e$ such that $w(x_1,\ldots, x_e)$ has that property. 

Perhaps surprisingly, some effort is required to show that generic tuples in $\Omega$ generate an infinite group (and therefore are not contained in a subfield subgroup of $G$). The reason for this difficulty is that the property of having infinite order, or of generating an infinite subgroup, is not an open condition. However if we assume $G$ is defined over an uncountable algebraically closed field, then having infinite order is a nonempty generic condition. 

\begin{lemma} 
Let $k$ be an uncountable algebraically closed field. Then generic tuples in $G^e$ have infinite order. Furthermore if some $\omega\in\Omega$ has infinite order, generic tuples in $\Omega$ share this property.  
\end{lemma} 
\begin{proof} 
The set $\Theta_n=\{x\in G \ |\ x^n=1\}$ of elements of order dividing $n$ is closed in $G$. Over an uncountable field, $(\bigcup_{n\geq1} \Theta_n)^c$ is a generic subset of $G$. Applying Lemma 2.8, generic tuples in $G^e$ have infinite order. If some $\omega\in\Omega$ has infinite order, then $\Omega\cap \Theta_n$ is a proper closed subvariety of $\Omega$ for each $n$.
\end{proof} 

Furthermore over an uncountable algebraically closed field, the existence of a single tuple $\omega\in\Omega$ topologically generating $G$ implies a generic subset of $\Omega$ possess the same property. 

\begin{lemma} Let $G$ be defined over an uncountable algebraically closed field $k$. If there is a tuple $\omega\in\Omega$ topologically generating $G$, then generic tuples in $\Omega$ generate $G$ topologically. 
\end{lemma} 
\begin{proof} First assume $G$ has characteristic $p=0$. Then Lemmas 2.5 and 2.6 and Remark 2.7 show that the existence of a single tuple $\omega\in\Omega$ topologically generating $G$ implies an open dense subset of tuples in $\Omega$ share this property (this is also true over any algebraically closed field). 
 
 In positive characteristic, Lemma 2.6 shows there is a proper closed subvariety of $G^e$ outside of which a tuple $\omega\in G^e$ generates a dense subgroup of $G$ if and only if it generates an infinite subgroup of $G$. By Lemma 2.9, the existence of a single tuple $\omega\in\Omega$ generating an infinite subgroup implies that generic tuples in $\Omega$ share this property.\end{proof} 

Now recall $G$ acts primitively on a $kG$-module $V$ if it acts irreducibly and preserves no non-trivial direct sum decomposition of $V$. The following lemmas establish that the set of tuples in $G^e$ generating a primitive subgroup is an open subset of $G^e$. 

\begin{lemma}   Let $G$ be a subgroup of $GL(V)$, with $V$ an $n$-dimensional vector space.  Then $G$ acts primitively on $V$ if and only if
\begin{itemize}
 \item[$(i)$] $G$ acts irreducibly on $V$, and
 \item[$(ii)$] For each proper divisor $d$ of $n$,  no subgroup $H$
with $[G:H]=n/d$ fixes a $d$-dimensional subspace of $V$.
\end{itemize}
\end{lemma}

\begin{proof}  To start we may assume that $G$ acts irreducibly on $V$ or else there is nothing
to prove. If $G$ acts imprimitively and irreducibly, then
$V=V_1 \oplus \ldots \oplus V_{n/d}$ with $n/d > 1$, and $G$ permutes the $V_i$ transitively.  If $H$ is the stabilizer of $V_1$, then $[G:H]=n/d$.

Conversely, suppose there is a subgroup $H$ of $G$ such that $[G:H]=n/d$, $d < n$, and $H$ fixes a $d$-dimensional subspace $V_1$ of $V$.   For $1 \le i \le n/d$, let $V_i$ be the conjugates
of $V_1$ under $G$.  Since $G$ acts irreducibly,  $V=V_1 \oplus\ldots\oplus V_{n/d}$, and  $G$ is imprimitive.
\end{proof}

Next let $n$ be any positive integer and fix a divisor $d$ of $n$. Since the free group $F_e$ is finitely generated, it has only finitely many subgroups of index $d$. Label these subgroups  $H_1(d),\ldots,H_t(d)$. For each such $H_i(d)$, $1\leq i\leq t$, choose a collection of generators
$w_{i1},\ldots, w_{is}$ of $H_i(d)$. 

\begin{lemma}   Let $G$ be a simple subgroup of $GL(V)$, with $V$ an $n$-dimensional vector space defined over an algebraically closed field.  The set of tuples $\omega\in\Omega$ 
acting primitively on $V$ forms a Zariski open subset of $G^e$.
\end{lemma}

\begin{proof}  
To begin, note the set of tuples $(x_1,\ldots,x_e)\in G^e$ such that $\langle x_1,\ldots,x_e\rangle$ fixes a common $d$-dimensional subspace of $V$ is the same as the set of $(x_1,\ldots,x_e)\in G^e$ such that $\overline{\langle x_1,\ldots,x_e\rangle}$ fixes a common $d$-dimensional subspace of $V$. By Lemma 2.5, the set of tuples in $G^e$ topologically generating
a subgroup acting reducibly on $V$ is closed.  

Now suppose $(g_1, \ldots, g_e)\in G^e$ topologically 
generates a subgroup $J$ of $G$ that acts irreducibly and imprimitively on $V$. Then by Lemma 2.11, for some proper divisor $d$ of $n$ there is a subgroup
$J_1\subset J$ such that $[J:J_1]=n/d$, and $J_1$ fixes a $d$-dimensional subspace of $V$.  Choose
a surjection $\varphi:F_e\rightarrow J$ such that $\varphi^{-1}(J_1)= H_i(d)$ for some $1\leq i \leq t$.  Then  
$\langle w_{i1}(g_1, \ldots, g_e),\ldots, w_{is}(g_1, \ldots, g_e)\rangle$ fixes a common $d$-dimensional subspace of $V$. 

By Lemma 2.5, the set of tuples $(x_1,\ldots,x_e)\in G^e$ such that the closure of $\langle w_{i1}(x_1, \ldots, x_e),\ldots, w_{is}(x_1, \ldots, x_e)\rangle$ fixes a common $d$-dimensional subspace of $V$ is closed in $G^e$. Then considering each subgroup  $H_1(d),\ldots,H_t(d)$ for each divisor $d$ of $n$, we find the set of tuples $(x_1,\ldots,x_e)\in G^e$ topologically generating an
imprimitive subgroup is closed.
\end{proof}
\\
\begin{remark} When $G$ is a simple subgroup of $GL(V)$, $e>1$, and $k$ is not algebraic over a finite field, there are tuples $(x_1,\ldots,x_e)\in G^e$ topologically generating $G$ (see Corollary 3.4 in \cite{guralnick}). So the open subset described above is nonempty. Of course this set is empty when $e=1$.

\end{remark}

To conclude this section, let us focus on  conditions for topological generation for specific classical groups. These results will be used in the proof of Theorem 1.1, and later in \cite{burness1}. In what follows, assume $G$ is a simply connected simple algebraic group of classical type, and $V$ is  the natural module for $G$. Recall for a conjugacy class $C_j$ of $G$ that $\gamma_j$ is the dimension of the largest eigenspace on $V$ of a representative of $C_j$, and $\Omega= C_1\times\cdots \times C_e$.
Set $n= \dim  V$. 
The following lemma gives a necessary condition for topological generation. 

\begin{lemma} 
If $\sum_{i=1}^e\gamma_i> n(e-1)$, then no tuple $\omega\in\Omega$ topologically generates $G$. 
\end{lemma} 
\begin{proof}   For any element $x\in C_i$, let $E_{\alpha_i}$ be a $\gamma_i$-dimensional eigenspace of $x$ on $V$. We claim 
$$\dim  (\bigcap_{i=1}^e E_{\alpha_i})= \sum_{i=1}^e\gamma_i- n(e-1).$$
\par
 This is easily checked by induction on the number of classes $e$. If $e=2$ and $\gamma_1+\gamma_2>n$, then clearly $\dim\   (E_{\alpha_1}\cap E_{\alpha_2})= \gamma_1+\gamma_2-n$.  So assume $\sum_{i=1}^k\gamma_i> n(k-1)$, and 

$$\dim\  ( \bigcap_{i=1}^k E_{\alpha_i})= \sum_{i=1}^k\gamma_i - n(k-1).$$

Then 
$$\dim\  (\bigcap_{i=1}^k E_{\alpha_i}\cap E_{\alpha_{k+1}})= (\sum_{i=1}^k\gamma_i- n(k-1)) + \gamma_{k+1} -n.$$
Letting $e=k+1$, it follows that $$\dim\ (\bigcap_{i=1}^{e} E_{\alpha_i})>0$$ whenever $\sum_{i=1}^e\gamma_i> n(e-1)$. In this case, every tuple $\omega\in\Omega$ topologically generates a reducible subgroup on the natural module, and hence no tuple in $\Omega$ will generate $G$ topologically. 
\end{proof}

\par  Next assume $M$ is the stabilizer of a 1-dimensional totally singular subspace of $V$ (so $M$ is a maximal parabolic subgroup of $G$). For linear and symplectic groups, all 1-dimensional subspaces of $V$ are totally singular. Furthermore in these cases $G$ acts transitively on the relevant set of 1-spaces, so $G/M$ is isomorphic to projective space $P_1(V)$ as a variety. Hence to show there exists some tuple $\omega\in\Omega$ fixing no 1-dimensional subspace of $V$, it suffices to show \newline $$\Omega\not\subset X:=\bigcup_{g\in G}(M^g\times \ldots \times M^g)\subset G^e.$$ (For orthogonal groups, 1-dimensional non-singular subspaces must also be considered.)

\begin{lemma} Let $G=SL_n(k)$ or $Sp_n(k)$.  Then $\sum_{j=1}^e\gamma_j\leq n(e-1)$ if and only if there is a tuple $\omega\in\Omega$ fixing no $1$-dimensional subspace of $V$. 
\end{lemma} 
\begin{proof} To begin assume $\sum_{j=1}^e\gamma_j> n(e-1)$. For each conjugacy class $C_i$, pick some $x_i\in C_i$ and let $E_{\alpha_i}$ be a maximal dimensional eigenspace of $x_i$ on $V$.  By the argument given in Lemma 2.14, it follows that $\dim\  (\bigcap_{i=1}^{e} E_{\alpha_i})>0$. In particular, there exists a 1-dimensional subspace $E$ of $V$ such that $E\subseteq \bigcap_{i=1}^{e} E_{\alpha_i}$, and $\omega=(x_1,\ldots, x_e)$ fixes $E$. Hence every tuple $\omega\in\Omega$ fixes some 1-dimensional subspace of $V$. 

For the converse direction, assume $\Omega \subset X$ where $M$ is the stabilizer of a 1-dimensional subspace of $V$. Let $\varphi: G^e\times G/M\rightarrow G^e$ be the natural projection map. Since $M$ is parabolic, $G/M$ is a complete variety and $\varphi$ is a closed map. However, $\varphi(Z)=X$ where $$Z=\{(x_1,\ldots,x_e,v)\ | \ x_1v=\ldots=x_ev=v \}.$$
So $X$ is closed, and in particular  

$$\Omega\subset Y_M= \overline{\{(x_1,\ldots,x_e)\in\Omega\mid \langle{x_1,\ldots,x_e}\rangle\subset M^g \text{, for some } g\in G\}}.$$
Applying Lemma 2.3, 
$$\sum_{j=1}^e\dim \ (G/M)^{x_j}\geq (e-1)\cdot  \dim \ G/M.$$
Note the irreducible components of $(G/M)^x$ are projective
spaces associated with each eigenspace of $x$. In particular the dimension of $(G/M)^x$ is one less than
the dimension of the largest eigenspace of $x$ on $V$. 

So $\dim\  (G/M)^{x_j}= \gamma_j-1$, and 
  $\dim  G/M=n-1$. Applying this information to the above inequality yields $$(\sum_{j=1}^e \gamma_j)-e\geq (e-1)(n-1).$$ This  implies that $\sum_{j=1}^e \gamma_j> n(e-1)$. 
 \end{proof} 
 
\begin{remark}

Since $\Omega$ is an irreducible variety, Lemmas 2.5 and 2.15 imply that under the conditions given in Theorem 1.1, an open subset of tuples $\Omega_0\subset \Omega$ fixes no 1-dimensional subspace of $V$. 
\end{remark} 
\begin{remark}
No restrictions are placed on the conjugacy classes $C_1,\ldots,C_e$ of $G$ since only $1$-dimensional subspaces of $V$ are being considered. When moving to the stabilizers of higher dimensional subspaces, different arguments are required to compute $\dim\ (G/M)^{x}$ depending on whether the $x\in C$ are semisimple, unipotent, or mixed. This issue arises when providing inductive arguments to establish topological generation for symplectic and orthogonal groups.
\end{remark} 

Now pick a maximal torus $T$ of $G$ and let $Ad:G\to GL(Lie(G))$ be the adjoint representation. Note $Lie(G)$ embeds in $V\otimes V^*$, and the rank of $G$ is the dimension of $T$. An element $x\in G$ is $\it{regular}$ if the dimension of its centralizer is equal to the rank of the group, and $\textit{regular semisimple}$ if it is both regular and diagonalizable. Call a regular semisimple element of $G$ $\textit{strongly regular}$ if it has distinct eigenvalues on the root 
spaces of $Lie(G)$ with respect to $T$. In other words, $t\in T$ is strongly regular if for any two distinct roots $\alpha_i,\alpha_j$ of $G$, $\alpha_i(t)\ne \alpha_j(t)$. Note if an element $t\in T$ has eigenvalues $\lambda_1,\ldots, \lambda_n$ on $V$, then $t$ has eigenvalues $\lambda_i\cdot\lambda_j^{-1}$ on $V\otimes V^*$, for $1\leq i\leq n$, $1\leq j\leq n$. Hence if for every $1\leq i,j,k,l\leq n$, we have that   $\lambda_i\cdot\lambda_j^{-1}= \lambda_k\cdot\lambda_l^{-1}$ implies $i=j$ and $k=l$, then $t$ is strongly regular. 

The set of strongly regular elements is clearly open in $G$ (since the complement is defined by finitely many polynomial equations) and nonempty. Recall a tuple $(x_1,\ldots,x_e)\in\Omega$ is said to be strongly regular if there is a nontrivial word $w\in F_e$ such that $w(x_1,\ldots,x_e)$ has that property.  Applying Lemmas 2.8 and 2.9, being strongly regular and having infinite order are generic properties of $G^e$ (defined over an uncountable algebraically closed field).  Hence if some tuple $\omega\in\Omega$ is strongly regular of infinite order, then generic tuples in $\Omega$ share this property. 

Proposition 2.19 below provides conditions sufficient to ensure the existence of a tuple $\omega\in\Omega$ topologically generating $G=SL_n(k)$. In the next section it is shown that these conditions may in fact be satisfied.  We first recall a useful fact from the literature. Here, a maximal rank subgroup is a subgroup that contains a maximal torus of $G$.

\begin{lemma} 
Let $G=SL_n(k)$, $n\geq 3$, and let $H$ be a proper maximal rank subgroup of $G$. Then the connected component $H^{\circ}$ acts irreducibly on the natural module for $G$.
\end{lemma}

\begin{proof} 
This follows from Section 6 of \cite{lubeck}, and the main theorem of \cite{liebeck5}. 
\end{proof} 

\begin{proposition} Let $C_1,\ldots,C_e$ be noncentral conjugacy classes of $G=SL_n(k)$, where $k$ is an uncountable algebraically closed field, and $n\geq 3$. Assume 
\begin{itemize}
    \item[$(i)$] there is a tuple $\omega\in\Omega$ that acts irreducibly and primitively on the natural module $V$, and
    \item[$(ii)$] there is a tuple $\omega\in\Omega$ that is strongly regular of infinite order.
\end{itemize} Then there is a tuple $\omega\in\Omega$ topologically generating $G$.  
\end{proposition}
\begin{proof} Assume that conditions $(i)$ and $(ii)$ both hold. Then by Lemmas 2.5 and  2.11, a non-empty open subset of tuples $\Omega_0\subset\Omega$ share property $(i)$. As being strongly regular of infinite order is a generic property, a generic subset of tuples in $\Omega$ will satisfy condition $(ii)$. In particular, some $\omega\in\Omega$ satisfies both conditions. Let $H$ be the closure of the subgroup generated by this tuple.  Note  $Lie(H)$ is a sum of eigenspaces for a  strongly regular element $x\in H$, and so is a direct
sum $T_0\oplus N_{\alpha_1}\oplus\ldots\oplus N_{\alpha_j}$ of some collection of nontrivial weight spaces for $T$. Furthermore,
$ T_0\subseteq T$, where $T = C_G(x)$.  

Assume $Lie(H)\ne Lie(G)$. Then $Lie(H)$ is invariant under $T$, and
the closure of $\langle H, T\rangle$ is a proper maximal rank subgroup of $G$. However by Lemma 2.18 this is impossible. Hence $Lie(H)= Lie(G)$ which in turn implies $H = G$. So there exists a tuple $\omega\in\Omega$ topologically generating $G$. \end{proof}

\section{Proof of Theorem 1.1} 

We are now in a position to prove the main theorem. In this section, let $C_1,\ldots,C_e$ be noncentral conjugacy classes of $G=SL_n(k)$, where $k$ is an uncountable algebraically closed field of characteristic $p\geq 0$, and $V$ is the natural module for $G$. Pick a class $C$ from $C_1,\ldots,C_e$, and $x\in C$. Assume $x$ has eigenvalues  $\lambda_1,...,\lambda_t$ on $V$, and recall  $x$ is conjugate to a block diagonal matrix of the form

$$J= \left[\begin{array}{ccccc}
     J_{r_1}(\lambda_1)  & && &   \\
 & J_{r_2}(\lambda_2) &&&\\
  \\
    &  &  & &  J_{r_t}(\lambda_t) \\
   \end{array}\right]$$
where 
$$J_{r_i}(\lambda_i)= \left[\begin{array}{cccccc}
    \lambda_i & 1 & 0  &\ldots & 0      \\
  0 &  \lambda_i & 1 &\ldots  &0  \\
  \ldots \\
  \ldots \\
  0 & \ldots & 0  & \lambda_i & 1 \\
   0 & \ldots & 0  & 0 & \lambda_i \\
   \end{array}\right].$$
   \\
   We say that $J_{r_i}(\lambda_i)$ 
   is a \textit{Jordan block} of $x$, where $r_i$ indicates the size of the Jordan block, and $\lambda_i$ is the corresponding eigenvalue. Furthermore we say $J_{r_1}(\lambda_1)\oplus \ldots\oplus J_{r_t}(\lambda_t)$ is the \textit{Jordan form} of $x$. Note in this definition we are not assuming that the $\lambda_i$ are distinct. 

Finally, recall a conjugacy class $C$ of $G$ is \textit{quadratic} if a representative $x\in C$ has a degree two minimal polynomial. Note elements in a quadratic class $C$ have Jordan form $J_{r_1}(\lambda_1)\oplus \ldots\oplus J_{r_t}(\lambda_t)$, where $r_i\leq2$ for $1\leq i\leq t$, and $x$ has at most two distinct eigenvalues on $V$. 
   
   The ``only if'' direction of Theorem 1.1 follows from Lemma 2.14, in combination with the following well-known fact from linear algebra. 

\begin{lemma} Suppose $e=2$ and $C_1,C_2$ are quadratic conjugacy 
classes of $G$. If $n\geq 3$, then
every tuple $\omega\in\Omega$ generates a reducible subgroup of $V$. 
\end{lemma}
\begin{proof}
This follows from Lemma 3.14 in \cite{burness2}. 
\end{proof}

For the converse direction, the argument proceeds by induction on the dimension of the natural module. For both the base case and the induction, it is helpful to record some information about the maximal subgroup structure of $SL_n(k)$. This follows from more general work on the maximal subgroups of classical groups carried out by Aschbacher \cite{aschbacher}, Liebeck and Seitz \cite{liebeck5}, and others.

In outline, the results state that the positive-dimensional maximal closed subgroups of classical algebraic groups arise in a collection of geometric families $$\mathcal{C}= \mathcal{C}_1 \cup\mathcal{C}_2 \cup \mathcal{C}_3 \cup\mathcal{C}_4 \cup \mathcal{C}_6$$ and a collection $\mathcal{S}$ of almost simple irreducibly embedded subgroups. The geometric families can be described in a uniform fashion in terms of their action on the natural module $V$. Roughly speaking, the collection $\mathcal{C}_1$ comprises subgroups that stabilize a proper subspace of $V$, the members of $\mathcal{C}_2$ stabilize an orthogonal decomposition of $V$, members of $\mathcal{C}_3$ stabilize a totally singular decomposition of $V$, members of $\mathcal{C}_4$ stabilize a tensor decomposition of $V$, and members in $\mathcal{C}_6$ are normalizers of classical groups. The remaining positive dimensional maximal subgroups $\mathcal{S}$ are less easily described in a systematic fashion, but have the nice property that their connected components are simple groups modulo scalars. See \cite{liebeck5} for further details on these subgroup collections. 

Note that subfield subgroups are finite closed subgroups acting irreducibly and tensor indecomposably on the natural module. To exclude the possibility that a tuple $\omega\in\Omega$ topologically generates a group contained in a subfield subgroup, it clearly suffices to show that the tuple topologically generates an infinite subgroup. 

The following table lists the approximate  structures of the maximal geometric subgroups of $G=SL_n(k)$. Note the collection $\mathcal{C}_3$ described above is empty for $SL_n(k)$, and $P_m$ refers to a parabolic subgroup of $SL_n(k)$, which stabilizes an $m$-dimensional subspace of $V$. 
\\
\begin{center} 

\begin{tabular}{ |p{2cm}||p{2.5cm}|p{4cm}|p{2cm}|  }
 \hline
 \multicolumn{3}{|c|}{Table 1: The maximal geometric subgroups of $SL_n(k)$} \\
 \hline
Class & structure & conditions  \\
 \hline
 $\mathcal{C}_1$   & $P_m$    & $1\leq m \leq n-1$  \\
 $\mathcal{C}_2$&   $GL_m\wr \mathrm{S}_t$  & $n=mt, t\geq 2$   \\
 $\mathcal{C}_4$ & $GL_{n_1}\otimes GL_{n_2}$ &  $n=n_1n_2$, $2\leq n_1<n_2$\\
     &$(\otimes^{t}_{i=1}GL_m).\mathrm{S}_t$ &  $n=m^t$, $m\geq 3,t\geq 2$ \\
 $\mathcal{C}_6$ &   $Sp_n$  & $n$ even \\
 & $SO_n$  & $p\ne 2$    \\

 \hline
\end{tabular}

\end{center} 
\
\\

\begin{theorem} 
For $G=SL_n(k)$, $n\geq 2$, the conjugacy classes of maximal closed geometric subgroups  of $G$ are listed in Table 1. 
\end{theorem} 

\begin{proof} 
See the main theorem of \cite{liebeck5}.
\end{proof} 

\begin{corollary} Let $M$ be a maximal positive dimensional closed subgroup of $SL_3(k)$.  Then one of the following holds:
\begin{itemize} 
\item[$(i)$]  $M$ is irreducible and primitive. In this case $M^{\circ} \cong SO_3(k)$.
\item[$(ii)$] $M$ is irreducible and imprimitive. In this case $M\cong  (GL_1(k)\wr S_3)\cap SL_3(k)$ is the normalizer of a maximal torus.
\item[$(iii)$] $M$ is reducible. In this case $M$ is the stabilizer of a line or a hyperplane on the natural module.
\end{itemize} 
\end{corollary}
\begin{proof} 
This follows from Theorem 3.2, and the fact  $\mathcal{S}$ is empty for $SL_3(k)$. 
\end{proof} 

For the base case of our argument, it will be helpful to record the dimensions of noncentral conjugacy classes in  $SL_3(k)$. 

\begin{lemma} Let $C$ be a noncentral conjugacy class of $SL_3(k)$. If $C$ is quadratic then $\dim   C=4$. Otherwise, $\dim  C=6$. 
\end{lemma} 
\begin{proof}
 Regular elements in $SL_3(k)$ have two dimensional centralizers, and hence if $C$ is a conjugacy class of regular elements, $\dim  C=6$. If $C$ is not a conjugacy class of regular elements, then $C$ is quadratic and
 representatives of $C$ either have Jordan form $x= J_2(\lambda_1)\oplus J_1(\lambda_2)$, or $x= J_1(\lambda_1)\oplus J_1(\lambda_1)\oplus J_1(\lambda_2)$. In either case, $\dim C_G(x)=4$ and $\dim  C=8-4=4$. 
\end{proof}
 
  The following theorem establishes  the base case $n=3$ of Theorem 1.1. For the case  $n=2$, see Theorem $4.5$ below. 
 \begin{theorem}

Let $C_1,\ldots,C_e$ be noncentral conjugacy classes of $G=SL_3(k)$, where $k$ is an uncountable algebraically closed field. Assume  $\sum_{j=1}^e \gamma_j\leq 3(e-1)$, and it is not the case that $e=2$ and $C_1$, $C_2$ are quadratic. Then there is a tuple $\omega\in\Omega$ topologically generating $G$.  
 
\end{theorem} 
\begin{proof}
First we show generic tuples in $\Omega$ generate an infinite subgroup of $G$.  In view of Lemma 3.4, we note that $\dim  \Omega\geq 10$. 

Assume $p > 0$, and $M$ is a subfield subgroup of $G$.  Pick a class $C$ from the list $C_1,\ldots,C_e$. Clearly $\dim \ (C\cap M)= 0$ and $\dim   G/M=8$. So  $\dim  \Omega \geq 10 > \dim  Y_M$, and hence $Y_M$ is a proper closed subvariety of $\Omega$. Now let  $M_1,M_2,\ldots$ be an enumeration of all subfield subgroups of $G$ up to conjugacy. Then  $\bigcup_{i\geq 1}  Y_{M_i}$ is a countable union of proper closed subvarieties of $\Omega$. So  generic tuples in $\Omega$ are contained in $(\bigcup_{i\geq 1}  Y_{M_i})^c$, and hence  generate an infinite subgroup. 

 Recall that  $\Delta=\prod_{j=1}^e(C_j\cap M)$. Working in arbitrary characteristic, it now only remains to show that  $$\dim  \Omega>  \dim  \Delta + \dim   G/M$$ for each maximal positive-dimensional closed subgroup $M$ of $G$. Once this is established, it follows that generic tuples in $\Omega$ generate an infinite group contained in no positive-dimensional maximal closed subgroup, and hence generic tuples in $\Omega$ generate $G$ topologically. Applying Corollary 3.3, it suffices to treat the following cases.

\begin{itemize}
\item[$(i)$] $M= SO_3(k)$. Pick a class $C$ from $C_1,\ldots,C_e$ such that $M\cap C$ is nonempty. Then $M \cap C$ is a finite union of $M$-classes, each of which has dimension $2$, so $\dim\  (M\cap C)=2$.  As $ \dim  G/M=5$,  it follows $\dim   \Omega>  2e +5$. 
\\
\item[$(ii)$] $M=N_G(T)$ is the normalizer of a maximal torus $T$. Pick a class $C$ from $C_1,\ldots,C_e$. Note $M^{\circ}=T$, so $\dim  G/M=6$ and $\dim\  (M\cap C)\leq 2$.  Applying Lemma 3.4, $\dim  \Omega> 2e + 6$ unless 
 $(a)$ $e=2$ and at most one of $C_1,C_2$ are non-quadratic, or $(b)$ $e=3$ and each $C_i$ is quadratic. 
 \par
If $C$ is quadratic, then either $M\cap C$ is finite or $C$ is a class of involutions. If $M\cap C$ is finite the desired inequality immediately holds. So assume $C$ is a class of involutions where  $M\cap C$ is positive-dimensional. Then $M\cap C$ is a union of classes of outer involutions in $M$
(there are two such classes if characteristic $p\neq 2$, and one when $p=2$). These outer involutions have a one dimensional centralizer in $M$. Hence if $C$ is quadratic, 
$\dim\  (M\cap C)\leq 1$. It follows that $\dim  \Omega> \dim  \Delta +\dim  G/M$.
   \\
\item[$(iii)$]  $M$ is reducible. Applying Lemmas 2.5 and 2.15  we see that generic tuples in $\Omega$ do not fix a 1-space or hyperplane (by considering the dual space).
 \end{itemize}
\end{proof}

Next we turn to the inductive argument. In order to make use of the inductive hypothesis, we will need to relate the conjugacy classes $C_1,\ldots, C_e$ of $SL_n(k)$ to classes $C_1',\ldots, C_e'$ of $SL_{n-1}(k)$. It will also be necessary to show the inequality $\sum_{j=1}^e \gamma_j'\leq (n-1)(e-1)$ holds, where $\gamma_j'$ is the dimension of the largest eigenspace of a representative of $C_j'$ on the natural module for $SL_{n-1}(k)$. 

To carry out this process, pick a class $C$ in $C_1,\ldots,C_e$, and some element $x\in C$.  Assume $x$ has Jordan form $J_{r_1}(\lambda_1)\oplus \ldots \oplus J_{r_t}(\lambda_t)$ on the natural module, and $\lambda_j$ is an eigenvalue corresponding to a maximal dimensional eigenspace $E_{\lambda_j}$. Finally, let $J_{r_j}(\lambda_j)$ be a Jordan block such that $r_j\leq r_i$ for all Jordan blocks $J_{r_i}(\lambda_i)$ with  $\lambda_i=\lambda_j$.

Now replace the Jordan block $J_{r_j}(\lambda_j)$ in the Jordan form of $x$ with a new Jordan block $J_{r_j-1}(\lambda_j)$. (If $r_j=1$, we remove the block $J_{r_j}(\lambda_j)$). Note an element $x'$ with this new Jordan form (say viewed in $GL_{n-1}(k)$) has determinant $\lambda^{-1}$. However, by scaling and restricting conjugation, we get a new conjugacy class $C^\prime$ of $SL_{n-1}(k)$. Every class $C_i^\prime$ can be extended to the class $C_i$ of $SL_n(k)$ in the obvious way, and rescaling does not effect generation properties. 

\begin{lemma} Let $C_1,\ldots,C_e$ be noncentral conjugacy classes of $SL_{n}(k)$, where $n\geq 3$ and $k$ is an uncountable algebraically closed field. Assume
\begin{itemize}
\item[$(i)$] $\sum_{i=1}^e \gamma_i\ \leq n(e-1)$, and 
\item[$(ii)$] it is not the case that $e=2$, and $C_1$, $C_2$ are quadratic. 
\end{itemize}
Finally, assume $C{_1}^\prime,\ldots,C_e^\prime$ are the classes of $SL_{n-1}(k)$ described above. Then $\sum_{i=1}^e \gamma_i'\leq (n-1)(e-1)$. 
\end{lemma}

\begin{proof}
Pick a class $C=C_i$, and $x\in C$. Assume $x$ has Jordan form $J_{r_1}(\lambda_1)\oplus \ldots\oplus  J_{r_t}(\lambda_t)$ on the natural module, and  $J_{r_j}(\lambda_j)$ is the Jordan block replaced in $x$ to obtain the Jordan form of elements in $C^\prime=C_i^\prime$. If $r_j=1$, then $\gamma_i^\prime=\gamma_i-1$. If $r_j>1$, then $\gamma_i^\prime= \gamma_i$. In the latter case, all Jordan blocks of $x$ have size greater than one. In particular if $\gamma_i^\prime=\gamma_i$ we have $\gamma_i^\prime\leq\frac{n}{2}$ with equality if and only if $C$ is quadratic.

We now argue by induction on the number of classes $e$. For the base case, assume $e=2$. Note by condition $(i)$, $\gamma_1+\gamma_2\leq n$. If either $C_1$ or $C_2$ contain a Jordan block of size one, then $\gamma^\prime_1+\gamma^\prime_2\leq\gamma_1+\gamma_2-1\leq n-1$. So assume $C_1$ and $C_2$ contain no Jordan blocks of size one. As it is not the case that both $C_1$ and $C_2$ are quadratic, we have  $\gamma_1^\prime+\gamma_2^\prime \leq n-1$. Now using the inductive hypothesis assume $\sum_{i=1}^{e-1}\gamma_i^\prime\leq (e-2)(n-1)$. Then $\sum_{i=1}^{e}\gamma_i^\prime \leq (e-1)(n-1)$.
\end{proof}

Let $\Omega^\prime= C_1^\prime\times\ldots\times C_e^\prime$ be the product of the classes of $SL_{n-1}(k)$ described in the above lemma. We may now complete the proof of the main theorem. \newline
\newline 
\emph{Proof of Theorem 1.1}
\newline 
\newline
Assume $n\geq4$. Using the inductive hypothesis, there exists a tuple $\omega^\prime\in \Omega^\prime$ topologically generating $SL_{n-1}(k)$. Extending $\omega^\prime$ to a tuple $\omega\in\Omega$, there is an  $\omega\in\Omega$ topologically generating a group which contains an isomorphic copy of $SL_{n-1}(k)$. In particular, this implies there is a tuple $\omega\in\Omega$ topologically generating a group which has a composition factor of dimension at least $n-1$ on the natural module. As $\omega$ generates a group acting irreducibly on space of dimension at least $n-1$, Lemma 2.5 implies a generic subset of tuples $\Omega_0\subset\Omega$ share this property. Furthermore Lemmas 2.5 and 2.15 imply a generic subset of tuples in $\Omega$ fix no $1$-space or hyperplane (by considering the dual space). Taken in combination, these facts imply generic tuples in $\Omega$ act irreducibly on the natural module. 

Second, the existence of a tuple $\omega\in\Omega$ topologically generating a group containing a copy of  $SL_{n-1}(k)$ implies there is a tuple $\omega\in\Omega$ topologically generating a group $N$ such that any finite index subgroup of $N$ acts irreducibly on a subspace of dimension greater than $n/2$. Recall the free group $F_e$ on $e$ letters has only finitely many subgroups of index $d$ for each divisor $d$ of $n$. Label these subgroups $H_1(d),\ldots, H_t(d)$, and for each $H_i(d)$ choose a collection of generators $w_{i1}, \ldots, w_{is}$ for $H_i(d)$. Repeating the argument given in Lemma 2.12, we find the set of tuples $(x_1,\dots,x_e)\in G^e$ such that $\langle w_{i1}(x_1, \ldots, x_e),\ldots, w_{is}(x_1, \ldots, x_e)\rangle$ fixes a common $d$-dimensional subspace of $V$ is closed in $G^e$. Then considering each subgroup  $H_1(d),\ldots,H_t(d)$ for each divisor $d$ of $n$, we find the set of tuples $(x_1,\ldots,x_e)\in \Omega$ having a finite index subgroup acting irreducibly on a subspace of dimension greater than $n/2$ is non-empty and open. It follows there is a tuple $\omega\in\Omega$ generating a group having properties $(i)$ and $(ii)$ in Lemma 2.11. Applying Lemma 2.12, generic tuples in $\Omega$ act primitively on the natural module. 

Similarly, using the inductive hypothesis there is a tuple $(x_1,\ldots,x_e)\in\Omega$ and word $w\in F_e$ such that $w(x_1,\ldots, x_e)$ has infinite order, and a tuple $(x_1,\ldots,x_e)\in\Omega$ and word $w\in F_e$ such that $w(x_1,\ldots, x_e)$ is strongly regular. Applying Lemmas 2.8 and 2.9, generic tuples in $\Omega$ must share both these properties. The finite intersection of generic subsets is a generic set, and over an uncountable field generic sets are non-empty.  Hence applying Proposition  2.19, there is a tuple $\omega\in\Omega$ topologically generating $G$. Finally, applying Lemmas 2.14 and 3.1 we see that the conditions given in the statement of Theorem 1.1 are in fact necessary.\qed

\section{Variations on the main theorem} 

In the statement of Theorem 1.1 it is assumed that $k$ is defined over an uncountable algebraically closed field. However once topological generation has been established over uncountable fields, more general results can be recovered under weaker hypotheses on $k$. 

\begin{lemma} 
Let $C_1,\ldots,C_e$ be noncentral conjugacy classes of a simply connected simple algebraic group $G$ defined over a field $k$ of characteristic zero. If there is a tuple $\omega\in\Omega$ topologically generating $G$, then an open dense subset of tuples in $\Omega$ topologically generates $G$.
\end{lemma} 

\begin{proof} Let $\Gamma$ be the set of tuples in $\Omega$ that generate a Zariski dense subgroup of $G$. Note that $\Gamma$ is defined over $k$. In characteristic zero, the assumption that $k$ is an algebraically closed field is not required in Lemma 2.6 and Remark 2.7 (see Theorem 4.1 of \cite{breuillard}). So $\Gamma$ is a nonempty open subset of $\Omega$.
\end{proof}

\begin{lemma} 
Let $C_1,\ldots,C_e$ be noncentral conjugacy classes of a simply connected simple algebraic group $G$, where $k$ is an algebraically closed field of characteristic $p>0$ that is not algebraic over a finite field.  If there is a tuple $\omega\in\Omega$ topologically generating $G$, then a dense subset of tuples in $\Omega$ topologically generate $G$. 
\end{lemma} 
\begin{proof} 
This follows from Theorem 2 of   \cite{burness}. 
\end{proof} 

\begin{corollary}
Let $C_1,\ldots,C_e$ be a collection of noncentral conjugacy classes of $SL_n(k)$, $n\geq3$, where $k$ is a field.  Assume $\sum_{i=1}^e\gamma_i\leq n(e-1)$, and it is not the case that $e=2$ and $C_1,C_2$ are quadratic. 
 \begin{itemize} 
 \item[$(i)$] If the characteristic  of $k$ is zero, then an open dense subset of tuples in $\Omega$ topologically generate $G$. 

 \item[$(ii)$] If $k$ is an algebraically closed field of positive characteristic  that is not algebraic over a finite field, then a dense subset of tuples in $\Omega$ topologically generate $G$. 
 \end{itemize} 

\end{corollary} 

\begin{proof} 
Theorem 1.1 and Lemma 4.1 imply $(i)$. Statement $(ii)$ follows from Theorem 1.1 and Lemma 4.2. 
\end{proof} 
\
\newline
We may now establish Theorem 1.2. 
\newline 
\newline 
\emph{Proof of Theorem 1.2} 
\newline
\newline
Corollary 4.3 establishes parts $(ii)$ and $(iii)$ of Theorem 1.2. Taken together, Lemma 2.10 and Theorem 1.1 show that part $(i)$ holds.\qed
\newline 

Finally, let us turn to the case $G = SL_2(k)$, which was excluded in Theorem 1.1. It is straightforward to prove an analogous result, although the statement differs slightly. Note if $C$ is a noncentral conjugacy class in $SL_2(k)$, then $\gamma=1$ and $C$ is quadratic. In particular it is always true that $\sum_{i=1}^e\gamma_i\leq n(e-1)$. Hence condition $(i)$ of Theorem 1.1 may be omitted,  and condition $(ii)$ becomes a statement about classes of involutions modulo the center.   

\begin{theorem} 
Let $C_1, C_2$ be noncentral conjugacy classes of  $G=SL_2(k)$, where $k$ is an uncountable algebraically closed field. There exists a tuple $\omega\in\Omega$ topologically generating $G$ if and only if it is not the case that $C_1$ and $C_2$ are classes of involutions modulo the center. 
\end{theorem} 

\begin{proof} Again it suffices to show  $\dim  \Omega> \dim  \Delta + \dim  G/M$, where $M$ is any subfield subgroup, or closed maximal positive-dimensional subgroup of $G$. First assume $M$ is a subfield subgroup $G(q)$ of $G$. Any noncentral conjugacy class $C$ of $SL_2(k)$ is two-dimensional. As $\dim\  (C\cap M)=0$ and $ \dim  G/M=3$, it follows that $\dim  \Omega= 4 > \dim  \Delta + \dim  G/M$. 

According to Theorem 3.2, each maximal closed positive dimensional subgroup of $G$ is either a Borel subgroup or the normalizer of a maximal torus. To start, assume $M$ is a Borel subgroup. If $C_1,C_2$ are semisimple, then elements in $C_j\cap M$ have a one dimensional centralizer, and  $\dim  \Omega> \dim \Delta + \dim  G/M$. 

A conjugacy class in $SL_2(k)$ is either semisimple or a scalar multiple of a unipotent class. By the above, if $C_1$ and $C_2$ are semisimple, then generic tuples in $\Omega$ are not contained in a Borel subgroup. So without loss of generality we may assume $C_1$ is unipotent (up to a scalar). Pick some $x_1\in C_1$. Since $C_1$ is a regular unipotent class, $x_1$ is contained in a unique Borel subgroup. Recall generic tuples in $\Omega$ have infinite order. Hence picking an element $x_2\in C_2$ of infinite order which does not live in this Borel subgroup, we have that $\langle x_1,x_2\rangle$ topologically generates $G$. Furthermore since $M$ is parabolic,

$$X=\bigcup_{g\in G} M^g\times M^g$$
is a closed subgroup of $G\times G$, so generic tuples in $\Omega$ topologically generate a group living in no conjugate of $M$.

Finally assume $M$ is the normalizer of a maximal torus. If $C$ is not a class of involutions modulo the center in $M$, then $\dim\  (C \cap M)=0$. Hence $\dim  \Omega> \dim  \Delta + \dim  G/M$, unless $C_1,C_2$ are both classes of involutions modulo the center. In this  latter case, $C_i\subset M^g$ for some $g\in G$. So $\Omega\subset X$, and no tuple in $\Omega$ will generate $G$ topologically. 
\end{proof}

The following generalization of Theorem 4.4 is the analogue of  Theorem 1.1 for $SL_2(k)$. 

\begin{theorem} 
Let $C_1,\ldots, C_e$ be noncentral conjugacy classes of $G=SL_2(k)$, where $k$ is an uncountable algebraically closed field. There exists a tuple $\omega\in\Omega$ topologically generating $G$ if and only if it is not the case that $e=2$ and $C_1$ and $C_2$ are classes of involutions modulo the center. \end{theorem} 
\begin{proof} Assume there is a conjugacy class $C$ in $C_1,\ldots,C_e$ that is not a class of involutions modulo the center. Without loss of generality assume $C\ne C_1$. By Theorem 4.4, there exists $x\in C$ and $x_1\in C_1$ such that $\langle x,x_1\rangle$ generates $G$ topologically. Furthermore, if $e=2$ and $C_1$ and $C_2$ are both classes of involutions modulo the center, then no tuple in $C_1\times C_2$ will  topologically generate $G$.
\end{proof} 
\section{Applications to generic stabilizers} 
 
We now turn our attention to  applications of Theorem 1.1, the first of which concerns linear representations of algebraic groups. In this section, let $G$ be a simple algebraic group defined over an algebraically closed field $k$,  and let $V$ be a finite dimensional rational $kG$-module. For any $x\in G$, let $$V^x=\{ v\in V \mid xv=v\}$$ be the vectors in $V$ fixed by $x$, and $$V(C)=\{ v\in V\mid xv =v \text{ for some } x\in C\}$$ be the vectors fixed by some element in the conjugacy class $C$. Finally let 

$$V^G=\{ v\in V \mid xv=v \text{ for all } x\in G\}$$
be the vectors in $V$ fixed by every element of $G$. 

Recall a generic stabilizer is trivial if there exists a non-empty open subvariety $V_0\subset V$ such that $G_v:=\{g\in G| \ gv=v\}$ is trivial for each $v\in V_0$. As discussed in Section 1, the cardinality of a generic stabilizer and in particular the question of when a generic stabilizer is trivial arises frequently in the context of invariant theory. 

 Let $G=SL_n(k)$. We will use Theorem 1.1 to  prove that a generic stabilizer is trivial when $\dim V$ is large enough. The relationship of this result to recent work of Garibaldi-Guralnick \cite{garibaldi} \cite{garibaldi2} \cite{garibaldi3} and Guralnick-Lawther \cite{guralnick2} \cite{guralnick2.5} is discussed in the first section.
 
 To begin, pick some conjugacy class $C$ of $G$, and assume $d$ elements of $C$ generate $G$ topologically. The first lemma yields a bound on the dimension of $V(C)$ in terms of $d$ and $\dim  C$.  
 
 \begin{lemma} 
 Let $G=SL_n(k)$, where $n\geq3$ and $k$ is an uncountable algebraically closed field. Let $V$ be a finite dimensional rational $kG$-module. Finally, let $C$ be a conjugacy class of $G$, and assume $d$ elements of $C$ generate $G$ topologically. Then $\dim V(C)\leq \frac{d-1}{d}\cdot \dim  V + \dim C$. In particular, $\dim  V(C)< \dim  V$, when $\dim  V> d\cdot \dim  C$.
 \ \end{lemma} 
 
 \begin{proof}
 Pick $x\in C$ and let $\varphi: G\times V^x\rightarrow V$ be the map $(g,v)\mapsto gv$. We claim $im(\varphi)=V(C)$. If $gv\in im(\varphi)$ then $xv=v$ and $gxv=gv$. In particular, $(gxg^{-1})gv=gxv=gv$. So $gv\in V(C)$. Conversely if $v\in V(C)$ there exists a $y\in C$ such that $y=gxg^{-1}$ and $yv=v$. Hence $g^{-1}v\in V^x$ and $\varphi(g, g^{-1}v)=v$, so $v\in im(\varphi)$. 
 
 To achieve the desired inequality, we compute a bound on the dimension of a fiber. Pick $v\in V(C)= im(\varphi)$. Then $xv=v$ for some $x\in C$. For any $h\in C_G(x)$,  we have $h^{-1}v=h^{-1}xv= xh^{-1}v$. So $h^{-1}v\in V^x$. Clearly $\varphi(h,h^{-1}v)=v$ for all $h\in C_G(x)$. Since $v\in im(\varphi)$ was arbitrary, every fiber has dimension at least $\dim \ C_G(x)$. In particular, 
 $$ \dim  V(C) \leq \dim  G + \dim  V^x- \dim  C_G(x).$$ 
 
 Now let $\gamma$ be the dimension of the largest eigenspace of a representative of $C$ acting on the natural module. Since $d$ elements of $C$ generate $G$ topologically, it follows from Theorem 1.1 that $d\gamma\leq (d-1)n$. This implies $\dim  V^x\leq \gamma\leq \frac{d-1}{d}\cdot  \dim  V$. Since dim $C= \dim  G- \dim  C_G(x)$, we have $$\dim  V(C)\leq \frac{d-1}{d}\cdot \dim  V + \dim  C.$$

 Finally, $\dim  V> d\cdot \dim  C$ if and only if $\dim  V > \frac{d-1}{d}\cdot \dim  V + \dim  C$, so $\dim  V(C)< \dim  V$ whenever $\dim  V> d\cdot\dim  C$.\end{proof}

Now let $\mathcal{P}$ be the set of noncentral conjugacy classes of $G$ containing elements of prime orders (and containing classes of arbitrary nontrivial unipotent elements in characteristic $p=0$). The following lemma from the literature allows us to conclude a generic stabilizer is trivial when $\dim   V(C)< \dim   V$ for all $C\in\mathcal{P}$. 

 \begin{lemma} 
Let $G$ be a simple algebraic group and $V$ be a finite dimensional rational $kG$-module such that $V^G=0$. Let $V(C)$ and $\mathcal{P}$ be as above. If $\dim   V(C)< \dim  V$ for all $C\in\mathcal{P}$, then a generic stabilizer is trivial. 
 \end{lemma} 
 
 \begin{proof} 
 This follows from Proposition 2.10 in \cite{burness2} and Lemma 10.2 in \cite{garibaldi4}.
 \end{proof} 
 
Finally, let $G=SL_n(k)$ and define
 $$\mathcal{P}^d:= \{C\in\mathcal{P}| \text{ \textit{d} elements of \textit{C} are required to generate $G$ topologically}\}$$
 $$\alpha_d:= \max \ \{\dim  C \ |\ C \in \mathcal{P}^d\}$$ $$\alpha:=\max \ \{ \alpha_d\cdot d\mid 2\leq d\leq n\}$$
 As noted in Section 1, it follows from Theorem 1.1 that if $n \geq 3$ and $C$ is a noncentral conjugacy class, then $G$ is topologically generated by $n$ elements in $C$.
 
Now assume $\dim  \ V>\alpha$. Then $\dim   V> d\cdot \dim   C$, for all $C\in\mathcal{P}$. Applying Lemma 5.1, dim $V(C)< \dim   V$ for every $C\in\mathcal{P}$, and hence by Lemma 5.2 a generic stabilizer is trivial. So to find an upper bound on the dimension of $V$ with a nontrivial generic stabilizer, it suffices to compute an upper bound on $\alpha$. 

\begin{theorem}
Let $G=SL_n(k)$, where $n\geq3$ and $k$ is an algebraically closed field. Let $V$ be a finite dimensional rational kG-module such that $V^G=0$. If $\dim  V > \frac{9}{4}n^2$, then a generic stabilizer is trivial.
 \end{theorem} 
 
  \begin{proof} To begin, let $k^\prime$ be an uncountable algebraically closed field extension of $k$. Note that if a generic stabilizer of $V(k)$ is nontrivial, then the set of points in $V(k)$ with a nontrivial stabilizer is dense. However $V(k)$ is dense in $V(k')$, so a generic stabilizer of $V(k^\prime)$ would then also be nontrivial. So without loss of generality we may assume $k$ is an  uncountable algebraically closed field. Furthermore, by the reasoning given prior to the statement of the theorem, it is sufficient to show $\alpha\leq\frac{9}{4}n^2$. 
  
  First we consider $\mathcal{P}^2$. Let $C$ be a regular semisimple conjugacy class in $G$. Then $C$ is not quadratic and by Theorem 1.1, $C\in\mathcal{P}^2$.  Regular semisimple classes are conjugacy classes of maximal dimension in $G$, so $\dim C=\alpha_2= n^2- n$. 
  
  For $d>2$, pick $C\in\mathcal{P}^d$ and $x\in C$.  By Theorem 1.1, $\gamma\geq \frac{d-2}{d-1}\cdot n$. It  follows that $C_G(x)$ must contain a copy of $GL_{\beta}(k)$ where $\beta = \lceil \frac{ n(d-2) } {d-1} \rceil $. It then follows from Proposition 2.9 in \cite{burness0} that $\alpha_d\leq n^2-\beta^2 -(n-\beta$), noting that $s = n-\beta$ in the notation of the reference. 

Finally, we compute an upper bound for $\alpha=\max\ \{ \alpha_d\cdot d\mid 2\leq d\leq n\}$. Note $\alpha_2= n^2-n$, and for $3\leq d\leq n$:
$$\alpha_d\leq n^2-\beta^2 -(n- \beta)< n^2-\beta^2\leq n^2-\dfrac{n^2(d-2)^2}{(d-1)^2}.$$
Furthermore,

$$d(n^2- \frac{n^2(d-2)^2}{(d-1)^2})= n^2\frac{d(2d-3)}{(d-1)^2}$$
and $\frac{d(2d-3)}{(d-1)^2}\leq \frac{9}{4}$ for $3\leq d\leq n$. So 
$$\alpha\leq\max  \ (2(n^2-n),n^2\frac{d(2d-3)}{(d-1)^2}) \leq \frac{9}{4}n^2.$$

\end{proof} 

Note that it is not claimed the above bound is sharp. However it is on the order of the best possible bound, as discussed in Section 1. For instance, there are examples where $G=SL_n(k)$, $V^G=0$, $\dim  V > n^2$ and
a generic stabilizer is nontrivial. 

\section{Applications to random generation}

We now turn to a final application of topological generation,  concerning the random generation of finite simple groups of Lie type. Let $H$ be a finite group and $I_t(H)$ be the elements of order $t$ in $H$. Assume $H$ contains elements of prime orders $r$ and $s$, and let
  $$P_{r,s}(H):=\frac{\mid \{ (x,y)\in I_r(H)\times I_s(H):\langle x,y\rangle =H\}\mid}{\mid I_r(H)\times I_s(H)\mid}$$ be the probability that two random elements of orders $r$ and $s$ generate $H$.  Now let $H_i$ be a sequence of finite simple groups of Lie type (of some fixed type) and assume $|H_i|\rightarrow\infty$ as $i\rightarrow\infty$. (Here the rank of the groups $H_i$ may vary, or the field they are defined over, or both). If $P_{r,s}(H_i)\rightarrow 1$ as $i\rightarrow\infty$, then groups of the relevant type are said to have \textit{random $(r,s)$-generation}. 
  
  The following result of Liebeck-Shalev \cite{liebeck2} shows that finite simple classical groups have random $(r,s)$-generation when the rank of the group is large enough (depending on the primes $r$ and $s$). In fact, explicit bounds on $f(r,s)$ are determined in \cite{stavrides}.

  \begin{theorem} 
 Let $(r,s)$ be primes with $(r,s)\ne (2,2)$. There exists a positive integer $f(r,s)$ such that if $H$ is a finite simple classical group of rank at least $f(r,s)$, then $$P_{r,s}(H)\rightarrow 1 \text{ as } |H|\rightarrow\infty.$$
\end{theorem}

The methods used to prove Theorem 6.1 in \cite{liebeck2} are probabilistic in nature.
We will use different tools, namely Theorem 1.1 and some algebraic geometry to establish the random $(r,s)$-generation 
of linear and unitary groups, where the rank is fixed and $r,s$ are independent of the rank. This gives a strengthening of Theorem 6.1 in these cases.

Let $G$ be a simply connected simple algebraic group defined over the algebraic closure $k$ of a prime field. Let $F$ or $F_q :G\rightarrow G$ be a Steinberg endomorphism, $G^F$ or $G(q)$ be the fixed points of $F$ on $G$, and $E(q):= E\cap G(q)$ for any subset $E$ of $G$. In particular, if $G=SL_n(k)$ we have $G(q)=SL_n(q)$ or $SU_n(q)$.  Finally, if $C_1, C_2$ are conjugacy classes of $G$ such that $\Omega(q):= C_1(q)\times C_2(q)\ne\emptyset$, let $P_{\Omega(q)}(G(q))$ be the probability that a random pair in $\Omega(q)$ generates $G(q)$. 

By Lemma 2.2 of \cite{guralnick3}, there exists a finite collection $\mathcal{U}=\{V_1,\ldots,V_d\}$ of finite-dimensional irreducible rational $kG$-modules such that every proper closed subgroup of $G$ that acts irreducibly on each $V_i\in\mathcal{U}$ is conjugate to some subfield subgroup $G(q)$ of $G$. 

In particular, for some fixed subfield subgroup $G(q_0)$ of $G$, there is a set $W_{q_0}$ such that both of the following hold: 

\begin{itemize}
    \item[$(i)$] $W_{q_0}$ consists of the tuples $(x_1,\ldots,x_e)\in G^e$ that act irreducibly on each module $V_i\in\mathcal{U}$.  
    \item[$(ii)$] Each tuple in $W_{q_0}$ generates either a Zariski dense subgroup of $G$, or a conjugate of some subfield subgroup $G(q)$ of $G$, with $G(q)\geq G(q_0) $. 
\end{itemize}

By Lemma 2.4 of \cite{guralnick3}, $W_{q_0}$ is defined over $\mathbb{F}_p$, and forms an open dense subset of $G^e$. Counting   $\mathbb{F}_q$-points on $\Omega\cap W_{q_0}$, the following theorem relates the existence of a tuple $\omega\in\Omega$ topologically generating $G$ to random tuples in $\Omega(q)$ generating the corresponding finite group of Lie type $G(q)$. 

\begin{theorem} 
 Let $G$ be a simply connected simple algebraic group defined over an algebraically closed field $k$ of  characteristic $p>0$. Let $C_1$ and $C_2$ be conjugacy classes of $G$, $\Omega=C_1\times C_2$,  and $Q$ be the set of powers $q=p^a$ such that $\Omega(q)=:C_1(q)\times C_2(q)\ne\emptyset$. The following are equivalent:
 \begin{itemize} 
\item[$(a)$]  If $k$ is uncountable, there exists a tuple $\omega\in \Omega$ topologically generating $G$. 

\item[$(b)$] If $k$ is uncountable, a generic subset of tuples in $\Omega$ topologically generate $G$.

\item[$(c)$] If $k$ is not algebraic over a finite field, there exists a tuple $\omega\in\Omega$ topologically generating $G$.

\item[$(d)$]  If $k$ is not algebraic over a finite field, a dense subset of tuples in $\Omega$ topologically generate $G$.

\item[$(e)$]  If $q\in Q$ is sufficiently large, then there exists a tuple $(x_1,x_2)\in\Omega(q)$ such that
$\langle x_1, x_2 \rangle= G(q)$.
 \item[$(f)$] $\lim_{q\in Q, q\rightarrow\infty} P_{\Omega(q)}( G(q))=1$.
 \end{itemize} 
\end{theorem}

\begin{proof} 
The equivalence of $(a)-(d)$ follows from Lemmas 2.10, 4.1 and 4.2 above. The equivalence of $(e)$ and $(f)$ is Corollary 5 in \cite{guralnick3}.  Hence it suffices to show $(a)$ is equivalent to $(e)$. 

To see that $(e)$ implies $(a)$, assume $k$ is an uncountable algebraically closed field. For any word $w(x,y)$, define a trace map

$$tr_{w}:C_1\times C_2 \rightarrow k, \ \  (x_1,x_2)\mapsto tr(w(x_1,x_2))$$
\newline 
where the natural representation is taken for classical groups and the adjoint representation is considered for exceptional groups. 

The proof of Theorem 3.1 in \cite{guralnick3} shows if there is a tuple $(x_1,x_2)\in\Omega(q)$ such that $\langle x_1,x_2\rangle=G(q)$ for sufficiently large $q$, there is a fixed word $w(x,y)$ such that $tr_w(x,y)$ is non-constant on $C_1\times C_2$. Furthermore by the proof of Corollary 2.8 $(ii)$ in \cite{guralnick3}, the trace map is a morphism mapping into a 1-dimensional variety $k$, so its image contains an open subset of its closure. Hence over any algebraically closed field $k$, its image attains all but finitely many values of $k$. In particular, there is a tuple $(x_1,x_2)\in\Omega$ such that $tr_w(x_1,x_2)$ is not in the algebraic closure of $\mathbb{F}_p$, and hence  $w(x_1,x_2)$ has infinite order. It then follows from Lemma 2.9 that generic tuples in $\Omega$ generate a group of infinite order. Since for some sufficiently large $q$ there is a tuple $(x_1,x_2)\in\Omega(q)$ such that $\langle x_1,x_2\rangle=G(q)$, $W_{q_0} \cap\Omega$ is a nonempty open subset of $\Omega$. The intersection of generic and open subsets is nonempty, so some tuple $\omega\in\Omega$ topologically generates $G$. 

Finally for $(a)$ implies $(e)$, assume there is a tuple $\omega\in\Omega$ topologically generating $G$, and $\dim   \Omega=s$. Then $W_{q_0}\cap \Omega\ne \emptyset$, and hence forms an open dense subset of $\Omega$. It follows $\Omega\backslash W_{q_0}$  is a proper closed subvariety of $\Omega$. Applying the Lang-Weil theorem \cite{lang} there is a positive constant $c$ such that for all sufficiently large $q\in Q$, $\mid\Omega(q)\mid> cq^s$. Hence there is a positive constant $c_1$ such that for sufficiently large $q$, $\mid \Omega(q)\backslash W_{q_0} \mid\leq c_1 q^{s-1}$. In particular for sufficiently large $q$, almost all tuples of $\Omega(q)$ are contained in $\Omega(q)\cap  W_{q_0}$. 

Next we count the proportion of pairs in $C_1(q)\times C_2(q)$ which are contained in some conjugate of $G(q_1)\times G(q_1)$, where $G(q_1)$ is a proper subfield subgroup of $G(q)$. By Lemma 2.5 of \cite{guralnick3},  $s= \dim   \Omega =\dim   G+ e$ for some $e\geq 1$. Now assume $q=p^a$,  $q_1=q^{\frac{1}{b}}$, and $d = \dim   G$. There exist positive constants $c_2$ and $c_3$
such that for large enough $q$, $\mid \Omega(q_1)\mid \leq c_2 q^{\frac{s}{b}}$, and  $\mid G(q):G(q_1)\mid \leq c_3 q^{d (1-\frac{1}{b})}$. Hence there is a positive constant $c_4$ such that for sufficiently large $q$, the proportion of pairs in $C_1(q)\times C_2(q)$ which are contained in some conjugate of $G(q_1)\times G(q_1)$ is bounded by 

$$\frac{\mid \Omega(q_1)\mid \mid G(q):G(q_1)\mid }{\mid \Omega(q)\mid}\leq c_4\big( \frac{q^{\frac{s}{b}} \cdot q^{d(1-\frac{1}{b})}}{q^s}\big)  =  c_4 q^{e(\frac{1}{b}-1)} $$
Summing over all positive divisors $b$ of $a$, we find the probability that a tuple in $\Omega(q)$ is contained in a proper subfield subgroup of $G(q)$ is at most  $c_4\sum_{b\mid a} q^{e(\frac{1}{b}-1)}$ and hence is bounded above by $O(q^{-\frac{e}{2}})$. So for any sufficiently large $q$, there is a tuple $(x_1,x_2)\in\Omega(q)\cap W_{q_0}$ not contained in any proper subfield subgroup $G(q_1)$ of $G(q)$. As $\langle x_1,x_2\rangle\subset G(q)$ and tuples in $W_{q_0}$ generate a group containing a subfield subgroup of $G$, it follows that $\langle x_1,x_2\rangle= G(q)$.
 \end{proof}

\begin{remark} Note while the results of Liebeck-Shalev \cite{liebeck2} are stated for finite simple classical groups, $G^F$ itself is not typically simple. This small discrepancy is easily resolved. If $G$ is a simply connected simple algebraic group, then $S=G^F/Z(G^F)$ is (almost always) a finite simple group. More generally, $[G^F,G^F]/Z([G^F,G^F])$ is simple except for a very small number of cases. Letting $\varphi: G^F\rightarrow S$ be the natural projection map, the center of $G^F$ is contained in the Frattini subgroup of $G^F$, and hence tuples $(x_1,x_2)\in\Omega(q)$  generate $G^F$ if and only if the corresponding pair $(\overline{x_1},\overline{x_2})$ generates $S$. (Of course, the orders of the elements may be different in $S$). 

An element $\overline{x}\in S$ has prime order $r$ if and only if the corresponding element $x\in G^F$ has prime power order $q=r^a$ for some $a\geq 1$. In particular, elements $\overline{x_1},\overline{x_2}\in S$ with prime orders $t$ and $u$ generate $S$ if and only if the corresponding elements $x_1$ and $x_2$ with prime power orders $r=t^a$ and $s=u^b$ generate $G$. Hence to prove random $(t,u)$-generation for primes $t$ and $u$ in the simple group $S$, it suffices to prove random $(r,s)$-generation in $G^F$ for prime powers $r$ and $s$. 
\end{remark} 
 \ 
\newline
\indent We now give the condition that will be used to establish the desired random $(r,s)$-generation result. To begin, fix prime powers $r$ and $s$, and let $G$ be a simply connected simple algebraic group. We say the pair of conjugacy classes $C'$ and $D'$ in $G$ are $\textit{bad}$ if the following hold:
\begin{itemize}
\item[$(i)$] $C'$ and $D'$ are classes of elements of orders $r$ and $s$ in $G$ and,
\item[$(ii)$] $\Omega'(q)= C'(q)\times D'(q)\ne\emptyset$ and, 
\item[$(iii)$] there is no tuple $\omega'\in\Omega'$ topologically generating $G$.
\end{itemize}

Similarly we say the pair of conjugacy classes $C$ and $D$ in $G$ are  \textit{good} if the following hold: 

\begin{itemize}
\item[$(i)$] $C$ and $D$ are classes of elements of orders $r$ and $s$ in $G$ and,
\item[$(ii)$]$\Omega(q)=C(q)\times D(q)\ne\emptyset$ and, 
 \item[$(iii)$] there is a tuple $\omega\in\Omega$ topologically generating $G$ and, 
  \item[$(iv)$] for any bad classes $C'$ and $D'$ of $G$, $\dim   \Omega'<\dim   \Omega$. 
 \end{itemize}
 
 \begin{lemma} Let $G$ be a simply connected simple algebraic group defined over an uncountable algebraically closed field of  characteristic $p> 0$. Let $Q$ be the set of powers $q=p^a$ such that $G(q)$ contains elements of prime power orders $r$ and $s$. If good classes exist for each $q\in Q$, then  $$\lim_{q\in Q, q\rightarrow\infty} P_{r,s}(G(q))=1.$$ 
\end{lemma} 
\begin{proof}
Fix $q\in Q$, and let $C_{1}, C_{2}$ be good classes corresponding to $q$. Next let $A_1,\ldots, A_l$ be all conjugacy classes of order $r$ elements in $G$ such that $A_i(q)\neq\emptyset$, and  $A=\bigcup_{i=1}^l A_i$. Similarly, let 
$B_1,\ldots, B_m$ be all conjugacy classes of order $s$ elements in $G$ such that $B_i(q)\neq\emptyset$, and  $B=\bigcup_{i=1}^m B_i$. Finally let $X= A\times B$, and $\Omega= C_1\times C_2$. 

Since $C_{1}$ and $C_{2}$ are good classes, $\dim   \Omega= \dim   X$. Hence there is a positive constant $c$ such that for sufficiently large $q$, $|X(q)| <c  |\Omega(q)|$. It follows that

$$ \lim_{q\in Q, q\rightarrow\infty}P_{\Omega(q)}(G(q))= \lim_{q\in Q, q\rightarrow\infty} P_{X(q)}(G(q))$$
where $P_{X(q)}(G(q))$ and $P_{\Omega(q)}(G(q))$ are the probabilities that  random tuples in $X(q)$ and $\Omega(q)$ generate $G(q)$. 

Now for each $q_i\in Q$, define $\Omega_i$ and $X_i$ as above by choosing appropriate good classes. Let $Q_i$ be the set of powers $q=p^a$ such that $\Omega_{i}(q)\ne\emptyset$. Theorem 6.2 implies $$\lim_{q\in Q_i, q\rightarrow\infty} P_{\Omega_{i}(q)}(G(q))=\lim_{q\in Q_i, q\rightarrow\infty} P_{X_{i}(q)}(G(q))= 1.$$
Since this holds for each $i$, it follows that   
$$\lim_{q\in Q,  i,q\rightarrow\infty} P_{X_i(q)}(G(q))= \lim_{q\in Q,  q\rightarrow\infty} P_{r,s}(G(q))=1.$$ 
\end{proof} 

Combined with Lemma 6.4, the following result completes the proof of Theorem 1.4.

 \begin{theorem} 
 Let $G=SL_n(k)$, $n\geq3$,  where $k$ is an uncountable algebraically closed field of positive characteristic. Assume $(r,s)\ne(2,2)$ are prime powers, and $F_q:G\rightarrow G$ is a Steinberg endomorphism such that $G(q)$ contains elements of orders $r$ and $s$. Then $G$ contains a pair of good classes of elements of orders $r$ and $s$.  
 \end{theorem}

 \begin{proof}
 First assume $G(q)=SL_n(q)$ and recall that
 $$|G(q)|= q^{\frac{n(n-1)}{2}}(q^2-1)(q^3-1)\cdots (q^n-1).$$
 Without loss of generality, assume $r> 2$. Pick any maximal dimensional conjugacy class $C$ of order $r$ elements in $G$, with $C(q)\neq\emptyset$. Note that $C$ is not quadratic. Let $\gamma_1$ be the dimension of the largest eigenspace of $C$. We claim $\gamma_1<\frac{n}{2}$. 
 
 If $r\mid q$ (so that the classes under consideration are unipotent) or $r\mid q-1$, then all classes of order $r$ elements are defined in $G(q)$, and $\gamma_1\leq \lceil \frac{n}{3} \rceil$. So assume $r\nmid q, q-1$, and pick $x\in C(q)$. Consider the action of $x$ on the natural module $V$ and define  
 $$l= \min \ \{j:\ r\mid q^j-1 \text{ and } 1\leq j\leq n\},$$ 
 $$m= \max \ \{j:\ r\mid q^j-1 \text{ and } 1\leq j\leq n\}.$$ 
 Note $x$ is semisimple, acts reducibly on a space of dimension at least $n-m$ with $n-m<\frac{n}{2}$ (since $r$ is a prime power, any element of order $r$ in $G(q)$ has a fixed space of dimension at least $n-m$), and has at most $m/l$ irreducible composition factors of dimension $l$, with $1< l\leq n$. A maximal dimensional
 conjugacy class of order $r$ elements in $G$ such that $C(q)\ne\emptyset$ does not have $m/2$ irreducible composition factors of dimension two. It follows that $\gamma_1\leq \max (n-m, \frac{n}{2}-1)< \frac{n}{2}$, and $C$ is not quadratic. 

Next let $D$ be any maximal dimensional class of order $s$ elements in $G$ such that $D(q)\neq\emptyset$. Let $\gamma_2$ be the dimension of the largest eigenspace of a representative in $D$. If $s>2$, then by the above we immediately have $\gamma_1+\gamma_2\leq n $. If $s=2$, and $D$ is a maximal dimensional conjugacy class of involutions, then $\gamma_2= \lceil \frac{n}{2} \rceil$. Again we have $\gamma_1+\gamma_2\leq n $. In both cases either $C$ or $D$ is not quadratic, and hence it follows from Theorem 1.1 that there is a tuple $\omega\in\Omega$ topologically generating $G$.

 So for any prime powers $(r,s)\ne (2,2)$ and any maximal dimensional conjugacy classes $C$ and $D$ of $G$ containing elements of orders $r$ and $s$, respectively,  with the property  $\Omega(q)\neq\emptyset$ for $\Omega=C\times D$, there is a tuple $\omega\in\Omega$ topologically generating $G$. Hence for each Steinberg endomorphism $F_q:G\rightarrow G$ defining $G(q)=SL_n(q)$, good classes of prime powers $(r,s)\ne (2,2)$ exist. 
 
 Finally, assume $F_q:G\rightarrow G$ is a Steinberg automorphism such that $G(q)=SU_n(q)$. Then 
 $$|G(q)|=q^{\frac{n(n-1)}{2}}(q^2-1)(q^3+1)\cdots(q^n-(-1)^n).$$ 
  However, if $r\mid q^l\pm 1$ and $r\nmid q^{j}\pm 1$ for $l<j\leq n$, then $l\leq \frac{n}{2}-1$. We may repeat the reasoning above to show good classes $C,D$ exist for prime powers $(r,s)\ne (2,2)$. 
 \end{proof} 
 \begin{remark}
 Note if $G=SL_2(k)$, then by Theorem 4.5 bad classes do not exist unless $C_1$ and $C_2$ are both classes of involutions modulo the center. Hence it follows from Theorem 6.5 and Remark 6.3 that $PSL_2(q)$ has random $(r,s)$-generation for primes $(r,s)\ne (2,2)$. 
 \end{remark}

\setcounter{tocdepth}{1}

\end{document}